\documentclass[a4paper]{article}

\usepackage[width=125mm]{geometry}
\usepackage{graphicx}
\usepackage[utf8]{inputenc}
\usepackage{amsmath, amsthm, amssymb, amsfonts,mathtools}

\usepackage{tikz-cd}

\usepackage{fourier}
\usepackage[french,english]{babel}
\usepackage{microtype}
\usepackage[backend=bibtex,style=alphabetic]{biblatex}

\newtheorem{theorem}{Theorem}[section]
\newtheorem{lemma}[theorem]{Lemma}
\newtheorem{proposition}[theorem]{Proposition}
\newtheorem*{proposition*}{Proposition}
\newtheorem*{theorem*}{Theorem}
\newtheorem{corollary}[theorem]{Corollary}
\newtheorem{definition}[theorem]{Definition}

\title{Geodesic convexity and  closed nilpotent similarity manifolds}

\author{Raphaël V.~{\sc Alexandre}\footnote{Institut de Math\'ematiques de Jussieu-Paris Rive Gauche, Sorbonne Université, 4 Place Jussieu, 75252 Paris Cédex, France. Email address: {\tt raphael.alexandre@imj-prg.fr}.}}

\AtBeginDocument{\maketitle}

\newcommand{\rT}{{\rm T}}
\newcommand{\cN}{{\mathcal N}}
\newcommand{\iI}{\mathopen{[}0\,,1\mathclose{]}}
\newcommand{\intv}[2]{\mathopen{[}#1\,,#2\mathclose{]}}
\newcommand{\eps}{\varepsilon}
\newcommand\R{\mathbf{R}}
\DeclareMathOperator{\Sim}{Sim}
\DeclareMathOperator\dd{d}
\DeclareMathOperator{\ad}{ad}
\DeclareMathOperator{\Ad}{Ad}

\DeclareMathOperator{\SO}{SO}
\DeclareMathOperator{\Aff}{Aff}
\DeclareMathOperator{\Aut}{Aut}

\DeclareMathOperator{\rO}{O}

\bibliography{ref.bib}

\begin{document}

\begin{abstract}
Some nilpotent Lie groups possess a transformation group analogous to the similarity group acting on the Euclidean space. We call such a pair a nilpotent similarity structure. It is notably the case for all Carnot groups and their dilatations.
We generalize a  theorem of Fried:  closed manifolds with  a nilpotent similarity structure are either complete or radiant and, in the latter case, complete for the structure of the space deprived of a point. 
The proof relies on a  generalization of convexity arguments in a setting where, in the coordinates given by the Lie algebra, we study geodesic segments instead of linear segments. 
We show  classic consequences for closed manifolds with a geometry modeled on the boundary  of a rank one symmetric space.
\end{abstract}

\section{Introduction}

Let $\Sim(\R^n)$ be the group of similarities of $\R^n$. That is to say, an element of $\Sim(\R^n)$ is of the form $x\mapsto\lambda P(x)+c$ where $\lambda>0$ is  a dilatation factor, $P\in \rO(n)$ is a rotation and $c\in \R^n$ describes a translation. Similarity manifolds are smooth manifolds equipped with an atlas of charts with values in $\R^n$ and transition maps in $\Sim(\R^n)$.

 More generally, a \emph{nilpotent similarity} group of transformations, $\Sim(\cN)$, acting on a nilpotent Lie group $\cN$, is described by the transformations of the form $x\mapsto \lambda P(x)+_\cN c$, where the addition $+_\cN$ is the group law of $\cN$, $\lambda>0$ is again a dilatation that can be seen through a coordinate system of the nilpotent Lie algebra  as $\lambda(x_1,\dots,x_n) = (\lambda^{d_1}x_1,\dots,\lambda^{d_n}x_n)$ with $d_i\geq 1$, $P$ is a rotation for the choice of the Euclidean metric associated to the coordinate system $(x_1,\dots,x_n)$ and $c\in \cN$ is again representing a translation. This larger class of geometries includes for example Carnot groups with their dilatations.

Fried~\cite{Fried} showed that every closed (real) similarity manifold is either Euclidean or radiant and therefore covered by a Hopf manifold. 
In other terms, if a closed similarity manifold is not complete, then it is radiant (the holonomy fixes a point and the developing map avoids this point).
It was also proven by alternative analytic methods in an independent work of Vaisman and Reischer \cite{Vaisman}. A generalization has been made for the Heisenberg group  by Miner \cite{Miner}.

The main obstruction to generalize Fried's theorem for any nilpotent similarity structure  is the need for a generalized version of convexity and geodesic structures. 
Convexity arguments are used in Fried's  proof \cite{Fried}, but not in a very explicit form. When Miner generalized the result for the Heisenberg group \cite{Miner}, he pointed out the convexity arguments in use. He attributed them to Carrière  and to Fried. Indeed, Carrière \cite{Carriere} used such convexity arguments, and refered to Fried's article but also to Koszul \cite{Koszul}, Benzécri \cite{Benzecri} and Kobayashi \cite{Kobayashi}. In Miner's article \cite{Miner} the convexity arguments rely on the fact that geodesics in the Heisenberg group are straight lines, and therefore a  classic point of view on convexity remains: it is the property of containing interior straight lines. But with a general nilpotent space (i.e. with a nilpotent rank at least three), it is no longer true that geodesics are straight lines, and therefore a generalization of convexity arguments  makes sense.

We  understand here by “geodesics” a class of curves parametrized by the tangent space and invariant by left-translation. The most natural class of such geodesics in a Lie group is given by the integral lines of left-invariant vector fields. 
In Fried's and Miner's proofs, it is a major requirement that such a class of curves and that special subsets (convex subsets) are available.

We will say that a
 geodesic structure $(X,\exp)$ is the data of a smooth manifold $X$ together with an exponential map $\exp\colon \rT X \to X$, which verifies some additional hypotheses. Such an  exponential map might not come from a Riemannian structure. The pair $(X,\exp)$ is thought to be a geometric model where geodesics are fully defined on $\R$.
On a smooth manifold $M$, a $(X,\exp)$-structure will be the additional data of a local diffeomorphism from the universal cover, $\widetilde M$, of $M$ to $X$, called a developing map. It gives  a sense of what geodesics are on $\widetilde M$ (they are curves which are developed into geodesics of $X$) but might not be fully defined on $\R$.
At a point $p\in \widetilde M$ there is a subset $V_p\subset \rT_p \widetilde M$, called the visible set from $p$. Its elements are the vectors that can be taken as initial speeds of (fully defined) geodesic segments on $\widetilde M$. 

Geodesic convexity in $X$ will be the property for a subset to contain a geodesic segment for each pair of points, together with a stability property (a sequence of geodesic segments based at a fixed point and with a converging endpoint gives again a geodesic segment).  In $\widetilde M$, the convexity is for a subset to be injectively developed into a convex subset of $X$.

\paragraph{}
One of  the classic results that is generalized in section \ref{sec-2} will be the following.

\begin{theorem*}[\ref{thm-convexity}]
Let $M$ be a connected $(X,\exp)$-geodesic manifold.
The following properties are equivalent.
\begin{enumerate}
\item The developing map $D\colon\widetilde M \to X$ is a diffeomorphism.
\item For all $p\in \widetilde M$, the subset $V_p$ is convex and equal to $\rT_p\widetilde M$.
\item There exists $p\in \widetilde M$ such that $V_p$ is convex and equal to $\rT_p\widetilde M$.
\end{enumerate}
\end{theorem*}

This result should be compared with a more classic setting. If $(X,\exp)$ is the Euclidean space $\R^n$ together with its natural complete Riemannian structure, then it is well known that $D$ is a covering if, and only if, $V_p=\rT_p\widetilde M$. The injectivity of $D$ (i.e. when $D$ is a diffeomorphism) is then equivalent to the convexity of $V_p$ and we get the equivalence of the theorem. The same reasoning remains true if $\R^n$ is equipped with a complete pseudo-Riemannian structure. Note that in the classic case, geodesics are straight lines.

We will show that for some structures (that are said to be injective), in particular for the nilpotent similarity structures, $V_p=\rT_p\widetilde M$ implies that $V_p$ is convex.
Therefore, this theorem gives the first argument of Fried's theorem's proof: if the structure is not complete, that is to say $D$ is not a diffeomorphism, then at any point $p\in\widetilde M$, $V_p$ is not convex. 

This allows, as in the original version of Fried's proof, to find a maximal (geodesic convex) open ball in $V_p$, and to study its radius following $p$. The open balls of $\cN$, that are also geodesically convex,  are constructed in section \ref{sec-3}.  Their construction relies on a fine theorem of Hebisch and Sikora \cite{Hebisch}.

Section  \ref{sec-2} is also devoted to the other major convexity arguments. 
One of those convexity arguments that will be central is proposition \ref{prop-convG}. It allows to find vectors in $V_p$ with the initial data of a convex subset of $V_p$. In particular, with our previous open balls, the dynamic of an incomplete geodesic will describe $V_p$ more precisely: it will show that each $V_p$ is at least a half-space.

The last step, once we know that every $V_p$ is at least a half-space, consists of showing that the developing map will be a covering onto its image, and that the holonomy must be discrete. But a discrete holonomy implies the final result. This idea for the last argument is not due to Fried's original proof but appears in the survey of Matsumoto \cite{Matsumoto}.

Therefore, we will be able to prove the following generalization of Fried's theorem. The proof of theorem \ref{thm-nilpotent} will be given in section \ref{sec-4}.

\begin{theorem*}[\ref{thm-nilpotent}]
Let $M$ be a connected closed $(\Sim(\cN),\cN)$-manifold. If the developing map $D\colon\widetilde M \to \cN$ is not a diffeomorphism, then the holonomy group $\Gamma=\rho(\pi_1(M))$ fixes a point in $\cN$ and $D$ is in fact a covering onto the complement of this point.
\end{theorem*}

This theorem notably applies in the case of every Carnot group \cite{Pansu} and also Heinsenberg-type groups. This last family includes all hyperbolic boundary geometries \cite{Cowling}. In general, if $KAN$ is an Iwasawa decomposition of a semisimple group $G$, then $A$ can both be  seen as a maximal flat subspace of $G/K$ and as a dilatation group acting on $N$. When we see $N$ as the base space, we get a boundary geometry $MAN/MA$ where $M\subset N$ is the centralizer of $A$. When $A$ is of rank one,  the theorem holds. In section \ref{sec-3} we will discuss the framework of this theorem, examples and counter-examples.

Theorem \ref{thm-nilpotent} suggests that the study of nilpotent affine manifolds is close to the traditional  affine manifolds' study. By \emph{nilpotent affine manifold} we mean manifolds possessing a $(\Aff(\cN),\cN)$-structure, where $\Aff(\cN)=\Aut(\cN)\ltimes\cN$ is the group of the affine transformations of $\cN$.
In consequence, one could ask what becomes of the various central conjectures stated in the Euclidean affine geometry. For example, the Chern conjecture states that \emph{every  affine closed manifold has a vanishing Euler characteristic}. Does every nilpotent affine closed manifold also have a vanishing Euler class?

\paragraph{}
Finally in section \ref{sec-5}, we will show  classic consequences for the closed manifolds with a geometry modeled on the boundary  of a rank one symmetric space  $\mathbf H_{\mathbf F}^n$, $
\left( {\rm PU}_{\mathbf F}(n,1), \partial \mathbf H_{\mathbf F}^n \right)
$, 
where $\mathbf F$ can be the field of  real, or complex, or quaternionic or octonionic numbers. In the octonionic case, the only dimension considered is $n=2$.

\begin{theorem*}[\ref{thm-rank1}]
Let $M$ be a connected closed $\left({\rm PU}_{\mathbf F}(n,1),\partial \mathbf{H}^n_{\mathbf F}\right)$-manifold. 
If the developing map $D$ is not surjective then it is a covering onto its image. Furthermore,
 $D$ is a covering on its image if, and only if, $D(\widetilde M)$ is equal to a connected component of $\partial \mathbf{H}^n_{\mathbf F}-L(\Gamma)$, where $L(\Gamma)$ denotes the limit set of the holonomy group $\Gamma=\rho(\pi_1(M))$.
\end{theorem*}

\paragraph{Acknowledgement}This work is part of the author's doctoral thesis, under the supervision of Elisha Falbel. The author is sincerely indebted to E.~Falbel for the many discussions and encouragements given.


\section{Convexity}\label{sec-2}

\subsection{Geodesic structures}

\begin{definition}[Geodesic structure]\label{def-geostruct}
Let $X$ be a smooth manifold. We will say that $(X,\exp)$ is a \emph{geodesic structure} if a smooth map $\exp\colon\rT X \to X$ is such that for any $x\in X$ fixed, 
 $\exp_x \colon \rT_xX\to X$ is a surjective open map and a local diffeomorphism around $0\in \rT_xX$.

We call \emph{geodesic segment} a curve $\gamma \colon \iI \to X$ such that $\gamma(t) = \exp_x(tv)$ with $x=\gamma(0)$ and $v\in \rT_xX$.
A subset $C\subset X$ is said to be \emph{convex} if:
\begin{enumerate}
\item[a.] for every pair of points $(x,y)\in C$ there is a geodesic segment from $x$ to $y$ fully contained in $C$;
\item[b.] for any sequence $(\gamma_n)$ of geodesic segments all based in $p\in C$, if $\gamma_n(1)$ tends to $q\in \overline C$, then a subsequence of $(\gamma_n)$ tends to a geodesic segment $\gamma\colon\iI \to C$ from $p$ to $q$ such that $\gamma(t)\in C$ for $t<1$.
\end{enumerate}

We ask that $(X,\exp)$ verifies two more conditions.
\begin{enumerate}
\item Let $x\in X$ and $u\in \rT_xX$.
Let $0\leq s<1$ and $y=\exp_x(su)$. Then there exists $v\in \rT_yX$ such that for any $0\leq t<1-s$, $
\exp_y(tv) = \exp_x((s+t)u)$.
\item The space $X$ is locally convex: for any neighborhood of $x$, there exists a subset of the neighborhood that is open, convex and contains $x$.
\end{enumerate}
\end{definition}

Condition \emph{1.} ensures that a geodesic segment $\gamma(t)$ based in $\gamma(0)$ is also a geodesic segment based in $\gamma(s)$ for any $s\geq 0$ and $t\geq s$.

Note that  a geodesic segment $\gamma\colon \iI \to X$ is univocally defined by $\gamma(0)$ and the first derivative $\gamma'(0)$. Any curve $c\colon \intv 0T\to X$ can be parametrized by $\widetilde c\colon\iI\to X$ by taking $\widetilde c(t) = c(tT)$. We can therefore say that $c$ is a geodesic segment if $\widetilde c$ is a geodesic segment.

It is important to note that in general, $\exp_x\colon\rT_xX\to X$ is neither injective and nor a covering map.

\paragraph{Example 1}Riemannian and pseudo-Riemannian complete structures give a Levi-Civita connection $\nabla_uv$. A geodesic is a curve $\gamma$ such that $\nabla_{\dot\gamma}\dot\gamma=0$. Such curves, which are therefore solutions to a first order partial differential equation, are parametrized by $\rT X$. It gives an exponential map as required. Therefore, any complete Riemannian or pseudo-Riemannian manifold gives a geodesic structure.

\paragraph{Example 2}A left-invariant geodesic structure on Lie groups with a surjective exponential map is a way to get geodesic structures.

\begin{definition}
Let $G$ be a Lie group and $\mathfrak g$ its Lie algebra. The \emph{Maurer-Cartan form} is the $\mathfrak g$-valued 1-form $\omega_G\colon \rT G \to \mathfrak g$ defined by
\begin{equation}
\forall v_g\in \rT_g G, \;\omega_G(v_g) = (L_{g^{-1}})_*v_g,
\end{equation}
where $L_g$ denotes the left-translation $L_g(x)=gx$.
\end{definition}

Let $X$ be a Lie group. The Maurer-Cartan form $\omega_X\colon \rT X\to \mathfrak x$ defines left-invariant vector fields $V$ by the condition that $\omega_X(V)$ is constant. Take as geodesics the integral lines of such vector fields. Those integral lines are given by $\exp_x(tv)=L_x\exp_e(tv)$ and are exactly the maps such that $\exp_x(tv)^*\omega_X=v$ with base point $x$. Such a structure gives a geodesic structure if the exponential map of the Lie algebra to the Lie group is surjective.

On some homogenous spaces, the same construction can be adapted. It is the case for the reductive homogeneous spaces with surjective exponential map, and for those, geodesic segments are the projection of integral lines of horizontal vector fields.

\subsubsection*{Geodesic manifolds}

\begin{definition}[Geodesic manifold]
Let  $(X,\exp)$ be a geodesic structure and $M$ be a connected smooth manifold. We will say that $M$ is a \emph{$(X,\exp)$-geodesic manifold} if there exists a local diffeomorphism $D\colon\widetilde M\to X$ called a \emph{developing map}.
A \emph{geodesic segment} in $\widetilde M$ is a smooth curve $\gamma\colon\iI\to \widetilde M$ such that $D(\gamma)\colon\iI\to X$ is a geodesic segment of $(X,\exp)$.
\end{definition}

Note that since $D$ is a local diffeomorphism, a geodesic segment $\gamma\colon\iI \to \widetilde M$ is univocally defined by $\gamma(0)$ and some vector $\gamma'(0)=v\in \rT_x\widetilde M$ given by $v = (\dd D_{\gamma(0)})^{-1}(u)$ such that we have $(D\circ\gamma)(t)=\exp_{D(\gamma(0))}(tu)$. 

\begin{definition}
If $\gamma\colon \iI \to \widetilde M$ is geodesic, then we denote $\gamma(t)$ by $\exp_{\gamma(0)}(tv)$ with $v=\gamma'(0)$.
\end{definition}

To distinguish the exponential map of $X$ from the last on $\widetilde M$, we will denote sometimes $\exp^X$ and $\exp^{\widetilde M}$.

It is worth to note that $D\colon\widetilde M \to X$ is independent from the choice of the geodesic structure $(X,\exp)$ but only dependent on $\widetilde M$ and $X$. In general, a developing map is hard to construct. A $(G,X)$-structure in the sense of Thurston~\cite{Thurston} is a way to give such a developing map.
Also according to a theorem of Whitehead~\cite{Whitehead}, any open manifold of dimension $n\leq 3$  has a local diffeomorphism with $\mathbf R^n$.
In the framework of Cartan geometries, a developing map corresponds to a flat Cartan connection.

\begin{proposition}[Definition of the visible set]
Let $p\in \widetilde M$. There exists a unique subset $V_p\subset \rT_p\widetilde M$ that is a neighborhood of $0$, star-shaped, maximal and on which $\exp_p\colon V_p\to \widetilde M$ is well defined. This set is called the \emph{visible set} of (or from) $p$.
\end{proposition}
\begin{proof}
By assumption, the space $X$ is locally convex, and the developing map is a local diffeomorphism. Hence a maximal $V_p$ is defined and non empty. It is also naturally star-shaped and must be a neighborhood of $0$ since the developing map is a local diffeomorphism.
\end{proof}

For all $p\in \widetilde M$ and $v\in V_p$, then by definition,
\begin{equation}
\forall t\in \iI, \; D\left(\exp_p^{\widetilde M}(tv)\right) = \exp_{D(p)}^{X}(\dd D_p(tv)).
\end{equation}
Therefore  if for $p\in \widetilde M$ we take $v\in\partial V_p-V_p$, then there can not exist a vector $u\in V_p$ such that for all $t\in \iI$, $D(\exp_p^{\widetilde M}(tu)) = \exp_{D(p)}^{X}(t\dd D_p (v))$.

\begin{definition}[Convexity]
A subset $C\subset \widetilde M$ is said to be \emph{convex} if the developing image $D(C)$ is convex and if the developing map restricted to $C$ is injective.
A subset $C_p\subset V_p$ for $p\in \widetilde M$ is said to be \emph{convex} if $\exp_p(C_p)$ is convex.
\end{definition}
 This injectivity hypothesis may seem strong. However, note that the developing map could be injective on $\exp_p(C_p)$ even if $\exp_p$ is not injective.

Since the developing map is a local diffeomorphism, and since $X$ is locally convex, we get the following lemma.

\begin{lemma}
The space $\widetilde M$ is locally convex: for every point $p\in \widetilde M$,  there exists an arbitrary small open neighborhood of $p$ that is convex in $\widetilde M$.\qed
\end{lemma}

Note that if $\eta\colon\iI\to X$ is a geodesic segment lifted to $\gamma\colon\iI \to \widetilde M$, that is to say $D(\gamma)=\eta$, then (since $D$ is a local diffeomorphism) $\gamma$ is unique as soon as $\gamma(0)$ is prescribed.

\begin{lemma}\label{lem-convvisible}
Let $C\subset \widetilde M$ be convex. For any $p\in C$ we have $C\subset \exp_p(V_p)$.
\end{lemma}
\begin{proof}
By convexity, $D|_C\colon C\to D(C)$ is a diffeomorphism.
Let $p\in C$ and $q\in C$. If $\gamma$ is a geodesic segment from $D(p)$ to $D(q)$, then $D^{-1}(\gamma)$ is  a geodesic segment from $p$ to $q$.
\end{proof}

\begin{proposition}
If for $p\in \widetilde M$, $V_p$ is convex and open, then $\exp_p(V_p)=\widetilde M$.
\end{proposition}
\begin{proof}
The set $\exp_p(V_p)$ is open. Indeed, $D(\exp_p^{\widetilde M}(V_p)) = \exp_{D(p)}^X(\dd D_p (V_p))$ is open since $\exp$ is an open map of $\rT X$ and $D$ is a diffeomorphism between $\exp_p(V_p)$ and $D(\exp_p(V_p))$. Therefore, it suffices to show that $\exp_p(V_p)$ is closed, since by connexity this implies $\exp_p(V_p)=\widetilde M$.
Let $q\in \widetilde M$ be in $\overline{\exp_p(V_p)}$. We show that $q\in \exp_p(V_p)$. 

By local convexity, there exists $C$ open and convex containing $q$. There exists a sequence $q_n\in C\cap \exp_p(V_p)$ such that $q_n\to q$. 

In the developing image, we can take a sequence $(\gamma_n)$ of geodesic segments from $D(p)$ to $D(q_n)$. By convexity of $D(\exp_p(V_p))$ in $X$, this sequence has a subsequence converging to a geodesic segment
$\gamma\colon\iI \to X$ such that: $\gamma(0)=D(p)$ and $\gamma(1)=D(q)$. We have furthermore, for $t<1$, $\gamma(t)\in D(\exp_p(V_p))$ and, for $t<1$ large enough, $\gamma(t)\in D(\exp_p(V_p)\cap C)$. Note that $\gamma(t)\to D(q)$ when $t\to 1$.

The geodesic segment $\gamma$ can be lifted to $\widetilde \gamma$ for $t<1$ in $\exp_p(V_p)$. Take $t_n \to 1$ an increasing sequence of times $0\leq t_n <1$. 
By  injectivity on $\exp_p(V_p)\cap C$ and since $\gamma(t_n)\to D(q)$, we have $\widetilde\gamma(t_n)\to q$. 
Therefore the lifting $\widetilde\gamma$ does not blow up when $t_n\to 1$ and
the natural compactification of $\widetilde \gamma$ by $\widetilde\gamma(1)=q$ lifts $\gamma$ for all $t\in\iI$. This shows that $q\in\exp_p(V_p)$.
\end{proof}

\begin{lemma}
Let $p\in \widetilde M$ and suppose that $V_p=\rT_p\widetilde M$. Then $D(\widetilde M)=X$.
\end{lemma}
\begin{proof}
By definition and the surjectivity of the exponential in $X$, if $V_p=\rT_p\widetilde M$ then
\begin{equation}
D\left(\exp_p^{\widetilde M}\left(\rT_p\widetilde M\right)\right)=\exp_{D(p)}^X\left(\dd D_p\left(\rT_p\widetilde M\right)\right) =\exp_{D(p)}^X\left(\rT_{D(p)}X\right)=X.\qedhere
\end{equation}
\end{proof}

\begin{theorem}\label{thm-convexity}
Let $M$ be a  connected $(X,\exp)$-geodesic manifold.
The following properties are equivalent.
\begin{enumerate}
\item The developing map $D\colon\widetilde M \to X$ is a diffeomorphism.
\item For all $p\in \widetilde M$, the subset $V_p$ is convex and equal to $\rT_p\widetilde M$.
\item There exists $p\in \widetilde M$ such that $V_p$ is convex and equal to $\rT_p\widetilde M$.
\end{enumerate}
\end{theorem}
\begin{proof}
Suppose that \emph{1} is true, we prove \emph{2}. For any $p\in \widetilde M$, it is clear that $V_p=\rT_p\widetilde M$ since any geodesic segment is lifted by $D^{-1}$. By definition, $V_p$ is also convex since $X$ is convex and $D$ is a diffeomorphism.

\emph{2} clearly implies \emph{3}.
We suppose that \emph{3} is true and we prove \emph{1}. By the preceding proposition, $\widetilde M = \exp_p(V_p)$ is convex and therefore $D$ is injective on $\widetilde M$. Furthermore by the preceding lemma, $D(\widetilde M)=X$.
\end{proof}

\paragraph{Convexity versus completeness}It is legitimate to investigate if $V_p=\rT_p\widetilde M$ implies that $V_p$ is convex. In a short moment, we will see that this is true for \emph{injective} structures. For now, consider the usual torus $\mathbf R^2/\mathbf Z^2$ and its universal cover $\pi\colon\mathbf R^2\to \mathbf R^2/\mathbf Z^2$. On the torus, consider the exponential map given by the straight segments of $\mathbf R^2$ projected on $\mathbf R^2/\mathbf Z^2$. (It is the natural Euclidean structure.) This gives a geodesic structure. The universal cover being a local diffeomorphism, we can see it as a developing map. Take $p\in \mathbf R^2$. Then $V_p=\rT_p\mathbf R^2$ but $V_p$ is not convex because $\exp_p(V_p)=\mathbf R^2$ is not injected into the torus. And indeed, $\pi$ is not a diffeomorphism.

If we no longer ask $D$ to be a diffeomorphism but only to be a covering map, it seems reasonable that geodesic completeness ($V_p=\rT_p\widetilde M$ for every $p\in \widetilde M$) is a sufficient condition.
But when $X$ is a simply connected space, which will be the case for us, it is equivalent to investigate when $D$ is a diffeomorphism.

\paragraph{Flat Cartan geometries and completeness}Suppose that $X$ is a homogenous space $G/H$ and is equipped with a geodesic structure. 
It is known in the Cartan theory (through the separated works of Ehresmann and Whitehead -- compare with \cite[p. 213]{Sharpe} for a more precise theorem), by very different techniques, that $\widetilde M$ is diffeomorphic to $G/H$ if and only if $\omega_D$ is complete (that is to say every $\omega_D$ constant vector field on $D^*G$ is complete).

On the other hand, Sharpe asks (see \cite[p. 184]{Sharpe}) what a good geometric interpretation of the completeness of a Cartan connection could be in terms of geodesics: “It would be very interesting to have a definition of completeness in terms of $M$ [...], something like completeness of geodesics.” This theorem gives such an interpretation in some cases. In our setting, the Cartan connection $\omega_D$ is complete if, and only if, $D$ is a diffeomorphism, and this is (by the preceding theorem) equivalent to a completeness condition on the geodesic structure on $\widetilde M$.

\subsubsection*{Injective structures}

An additional hypothesis can sometimes be made on $(X,\exp)$. It is notably the case when $X$ is an Hadamard space.

\begin{definition}
Let $(X,\exp)$ be a geodesic structure. We will say that it is an \emph{injective} geodesic structure if for any pair of points $(x,y)\in X\times X$, there exists a unique geodesic segment from $x$ to $y$.
\end{definition}

\begin{lemma}\label{lem-3geodesic}
Let $p,x\in \widetilde M$. Let $C$ be a convex subset of $\widetilde M$. Suppose that $p\in C$, $x$ is visible from $p$ and $D(x)\in D(C)$. Then $x\in C$.
\end{lemma}
\begin{proof}
Let $\gamma\colon\iI\to \widetilde M$ be a geodesic segment from $p$ to $x$. By injectivity of the geodesic structure, $D(\gamma)$ is the unique geodesic segment from $D(p)$ and $D(x)$ and is therefore by convexity entirely contained in $D(C)$. But $\gamma$ is also the unique lifting of $D(\gamma)$ based in $p$. In particular, the lifting of $D(\gamma)$ in $C$ by using the diffeomorphism $D|_C\colon C \to D(C)$ is again $\gamma$ and this shows that $\gamma(1)=x\in C$.
\end{proof}

\begin{proposition}
Let $C_1,C_2$ be two convex subsets with a non empty intersection. Then the developing map $D$ is injective on $C_1\cup C_2$.
\end{proposition}
\begin{proof}
Let $p\in C_1\cap C_2$. Suppose that for $q_1\in C_1$ and $q_2\in C_2$ we have $D(q_1)=D(q_2)$. Then there exists a unique geodesic segment from $D(p)$ to $D(q_1)=D(q_2)$. By convexity, this geodesic segment is simultaneously in $D(C_1)$ and in $D(C_2)$. By unicity of the lifted geodesic segment we have $q_1=q_2$.
\end{proof}

\begin{lemma}
Let $p\in \widetilde M$. The developing map $D$ restricted to the subset $\exp_p(V_p)$ is injective.
\end{lemma}
\begin{proof}
 Let $q_1,q_2\in \exp_p(V_p)$ and suppose that $D(q_1)=D(q_2)$. Let $\gamma_1,\gamma_2$ be two geodesic segments from $p$, to $q_1$ on one hand and to $q_2$ on the other hand. Then $D(\gamma_1)=D(\gamma_2)$ by the injectivity of the geodesic structure $(X,\exp)$. By unicity of the lifted geodesic segment $q_1=q_2$.
\end{proof}

\begin{proposition}\label{prop-convgeocomp}
Suppose that for $p\in \widetilde M$ we have $V_p=\rT_p\widetilde M$, then $V_p$ is convex.
\end{proposition}
\begin{proof}
We already know that $D(\exp_p(V_p))=X$ is convex. Furthermore the injectivity on $\exp_p(V_p)$ comes from the preceding lemma.
\end{proof}

This allows to state another version of theorem \ref{thm-convexity}.
\begin{corollary}[Theorem \ref{thm-convexity} for injective structures]
Let $M$ be a connected $(X,\exp)$-geodesic manifold, with $(X,\exp)$ an injective geodesic structure.
The following properties are equivalent.
\begin{enumerate}
\item The developing map $D\colon\widetilde M \to X$ is a diffeomorphism.
\item For all $p\in \widetilde M$, $V_p=\rT_p\widetilde M$.
\item There exists $p\in \widetilde M$ such that $V_p=\rT_p\widetilde M$.
\end{enumerate}
\end{corollary}

\subsection{Geodesic structures with compatible holonomy}

The preceding section addressed the question of the topology of $\widetilde M$. However, it is natural to ask what a geodesic structure on $M$ implies on $M$. To do so, we need to make the assumption that a transformation of the fundamental group $\pi_1(M)$ does not change the geodesic nature of a curve.

A \emph{$(G,X)$-structure} in the sense of Thurston~\cite{Thurston} is the pair of a smooth space $X$ together with a transitive group $G$ of analytic diffeomorphism acting on $X$. A manifold $M$ with a $(G,X)$-structure is called a \emph{$(G,X)$-manifold}. That is the case if $M$ has charts over $X$ with transitional maps in $G$\footnote{Our definition is here less general than Thurston's, since he allows $G$ to be only a pseudogroup. (Compare also with his book \cite{ThurstonBook}.)}.
If $M$ is a $(G,X)$-manifold, then we can construct a pair $(D,\rho)$ of the \emph{developing map} $D\colon\widetilde M\to X$ and the \emph{holonomy morphism} $\rho\colon\pi_1(M)\to G$. The developing map is a local diffeomorphism and those two maps are equivariant: $D(gx)=\rho(g)D(x)$ for any $x\in \widetilde M$ and $g\in\pi_1(M)$.

\begin{definition}Let $(G,X)$ be a geometrical structure in the sense of Thurston.
Let $(X,\exp)$ be a geodesic structure. Then $(G,X,\exp)$ is a \emph{geodesic structure with compatible holonomy}, if for any geodesic segment $\gamma\colon\iI\to X$ and any $g\in G$, the curve $g\gamma$ is again a geodesic segment.

If $M$ is a $(G,X)$-manifold in the sense of Thurston, then with no additional assumption, $M$ is a \emph{$(G,X,\exp)$-manifold}. The developing map $D$ is the developing map of  Thurston's structure. 
\end{definition}

\begin{proposition}\label{prop-convhol}
Let $M$ be a connected $(G,X,\exp)$-manifold.
\begin{enumerate}
\item If $\gamma\colon\iI\to \widetilde M$ is a geodesic segment based in $p$, then for any $g\in \pi_1(M)$, $g\gamma$ is a geodesic segment based in $gp$.
\item In particular, if for $p\in \widetilde M$ and $v\in\rT_p\widetilde M$, $\exp_p(tv)$ is only defined for $t<1$, then for any $g\in \pi_1(M)$, $g\exp_p(tv)$ is again only defined for $t<1$.
\item If $C\subset \widetilde M$ is convex, then for any $g\in \pi_1(M)$, $gC$ is again convex.
\end{enumerate}
\end{proposition}
\begin{proof}
The first two properties are clear, we prove the third. Let $C\subset \widetilde M$ be convex and let $g\in \pi_1(M)$. By equivariance, $D(gC)=\rho(g)D(C)$ is convex in $X$. It suffices to show that the developing map is injective when restricted to $gC$. Let $gx_1,gx_2\in gC$ and suppose that $D(gx_1)=D(gx_2)$ then $\rho(g)D(x_1)=\rho(g)D(x_2)$ and therefore $D(x_1)=D(x_2)$, implying $x_1=x_2$ and hence $gx_1=gx_2$.
\end{proof}

This shows that if we fix $x\in M$ and a curve $c\colon\iI \to M$ based in $x$, then if any lift of $c$ in $\widetilde M$ is geodesic, then this is in fact the case for any lift. Furthermore by this same proposition, if for $p\in \pi^{-1}(x)$, we have $V_p=\rT_p\widetilde M$, then it is again true for any $q\in \pi^{-1}(x)$. If this is the case, we denote $V_x=\rT_xM$.
In general, if we want to define $V_x\subset \rT_xM$ as $V_p$ for a $p\in\pi^{-1}(M)$, it depends on the choice of $p$, and a change in $p$ gives an isomorphism acting on $\rT_xM$.

The following corollary interprets theorem \ref{thm-convexity} in this framework.

\begin{corollary}\label{cor-convexity}
Let $M$ be a connected $(G,X,\exp)$-geodesic manifold.
The following properties are equivalent.
\begin{enumerate}
\item The developing map $D\colon\widetilde M \to X$ is a diffeomorphism.
\item For all $x\in M$, the subset $V_x$ is convex and equal to $\rT_x M$.
\item There exists $x\in M$ such that $V_x$ is convex and equal to $\rT_xM$.
\end{enumerate}
\end{corollary}
\begin{proof}
By theorem \ref{thm-convexity}, we only need to verify that propositions \emph{2} and \emph{3} correspond to propositions \emph{2} and \emph{3} of theorem \ref{thm-convexity}. But this follows from the preceding discussion and the preceding proposition (the fact \emph{3} concerning convexity).
\end{proof}

The next two propositions give topological properties on $M$.

\begin{proposition}\label{prop-trivconv}
Let $x\in M$ and $\pi\colon\widetilde M \to M$ be the universal cover. Then there exists an open trivializing neighborhood $U\ni x$ such that $\pi^{-1}(U)$ is a disjoint union of convex open subsets.
\end{proposition}
\begin{proof}
Let $U$ be any open trivializing neighborhood of $x$. Let $V$ be a connected open in $\pi^{-1}(U)$, containing $p\in\pi^{-1}(x)$. Since $\widetilde M$ is locally convex, we can reduce $V$ into $V'$ such that $V'$ is convex, open and contains $p$. Now, for any $g\in \pi^{-1}(M)$, $gV'$ is again convex and contains $gp\in \pi^{-1}(x)$. Since $\pi(gV')=\pi(V')\subset U$, this is an open trivializing neighborhood of $x$.
\end{proof}

\begin{proposition}
Let $x\in M$. If $V_x$ is convex and open, then any homotopy class with fixed endpoints of a continuous curve $c\colon\iI\to M$ with $c(0)=x$ is realized by a geodesic.
\end{proposition}
\begin{proof}
If $V_x$ is convex and open, then this means that there exists $p\in \pi^{-1}(x)$ such that $V_p$ is convex and open. But then $\exp_p(V_p)=\widetilde M$. Let $\widetilde c$ be the lift of $c$ in $\widetilde M$ based at $p$. Then there exists a geodesic segment from $p$ to $\widetilde c(1)$. This geodesic segment realizes the homotopy class of $c$ when projected by $\pi$ to $M$.
\end{proof}

This last proposition is to be compared with the Hopf-Rinow theorem. Indeed, suppose that $M$ is a compact 
Riemannian manifold with negative or null constant curvature. Then for the suitable $(X,\exp)$  (either the Euclidean or the hyperbolic space) and by compacity of $M$, the Hopf-Rinow theorem gives that for any $x\in M$, $V_x=\rT_xM$. This shows that $V_x$ is also convex by injectivity of the structure. The preceding proposition then shows  that any homotopic curve can be supposed geodesic. This representative geodesic is unique by injectivity of the geodesic structure $(X,\exp)$.

This phenomenon was already known, but with various other proofs involving other methods.

\subsubsection*{Injective structures with compatible holonomy}

Again, an additional injective hypothesis on the geodesic structure $(X,\exp)$ allows additional results. For example, it is not hard to reformulate corollary \ref{cor-convexity} as we did for theorem \ref{thm-convexity}: the condition $V_x=\rT_xM$ implies that $V_x$ is convex.
But the following construction is specific to an injective structure with compatible holonomy. It comes from the fact that in $X$, we can define $\exp_x^{-1}\colon X\to \rT_xX$, and that was not possible without the injectivity hypothesis.

This result appears in Miner's article \cite{Miner}. It is essential to the proof of Fried's theorem, and was used in Fried's original work \cite{Fried} without being stated independently.

\begin{proposition}\label{prop-convG}
Suppose that $(G,X,\exp)$ is an injective geodesic structure with compatible holonomy. Let $M$ be a connected $(G,X,\exp)$-manifold. Let $p\in \widetilde M$ and $g\in\pi_1(M)$. Suppose that $gp=\exp_p(u)\in \exp_p(V_p)$. We define
\begin{equation}
G = \dd D_p^{-1}\circ (\exp_{D(p)}^X)^{-1}\circ \rho(g)\circ \exp_{D(p)}^X\circ \dd D_p.
\end{equation}
The following properties are true.
\begin{enumerate}
\item We have $G(0)=u$.
\item If $v\in V_p\cap G^{-1}(V_p)$ and if $g\exp_p(v)\in \exp_p(V_p)$, then
\begin{equation}
\exp_p^{\widetilde M}(G(v))=g\exp_p^{\widetilde M}(v).
\end{equation}
\item Suppose that $C_p\subset V_p$ is a convex open subset containing $u$. If for $w\in \rT_p\widetilde M$ we have $G(w)\in C_p$ then $w\in V_p$.
\end{enumerate}
\end{proposition}
\begin{proof}
Property \emph{1} is immediate by definition. We show that \emph{2} is true. Let $v\in V_p\cap G^{-1}(V_p)$.
\begin{align}
D\left(g\exp_p^{\widetilde M}(v)\right) &= \rho(g) \exp_{D(p)}^X\left(\dd D_p(v)\right) \\
D\left(\exp_p^{\widetilde M}(G(v)\right) = \exp_{D(p)}^X\left(\dd D_p(G(v))\right) &= \rho(g) \exp_{D(p)}^X\left(\dd D_p(v)\right)
\end{align}
This shows that $\exp_p(G(v))$ and $g\exp_p(v)$ have the same developing image. The geodesic segment $\gamma\colon\iI \to X$ joining $D(p)$ to $D(g\exp_p(v))=D(\exp_p(G(v))$ is unique. Since the developing map is injective on $\exp_p(V_p)$, it suffices that both $g\exp_p(v)$ and $\exp_p(G(v))$ be visible from $p$,
and they are by assumption.

Now we show that \emph{3} is true. For $t\geq 0$ small enough, $tw$ belongs both to $V_p$ and $G^{-1}(V_p)$ since $G$ is continuous and $G(0)=u\in C_p$. Also, for $t$ small enough, $g\exp_p(tw)\in\exp_p(V_p)$ since it is a neighborhood of $gp$ by the existence of $C_p$. Therefore, for $t\geq 0$ small enough, we have $g\exp_p(tw)=\exp_p(G(tw))$.

Take a look at the linear segment $tw$ for $t\in \iI$. The map $\exp_{D(p)}\circ \dd D_p$ transforms this segment into a geodesic segment. The map $\rho(g)$ transforms this geodesic segment into a geodesic segment, say $c(t)\colon\iI\to X$. The maps $G$ and $c$ are related by $\exp_{D(p)}(\dd D_p(G(tw))) = c(t)$. But the endpoints of $c(t)$ and $D(\exp_p(G(tw)))$ are $D(\exp_p(G(0)))=D(gp)$ at $t=0$ and $D(\exp_p(G(w))$ at $t=1$, and both belong to $D(\exp_p(C_p))$. Hence, by unicity of the geodesic and by convexity of $\exp_p(C_p)$, the lift of $c$ based in $\exp_p(G(0))$ is exactly $\exp_p(G(tw))$. Hence, $\exp_p(G(tw))$ is well defined for all $t\in \iI$.

Now, the equation $g\exp_p(tw)=\exp_p(G(tw))$ was only true for $t$ small enough. But clearly, $g^{-1}\exp_p(G(tw))$ is always defined, and for $t$ small enough we have that $g^{-1}\exp_p(G(tw))=\exp_p(tw)$. Hence $g\exp_p(tw)$ is defined for every $t\in \iI$. It follows that $w\in V_p$.
\end{proof}

Note that  fact \emph{2} shows an equivariance between $\pi_1(M)$ and $\exp_M$. Fact \emph{3} allows to deduce visible vectors from a convex subset.

\section{Nilpotent similarity structures}\label{sec-3}

Let $\mathfrak g$ be a Lie algebra with the decomposition
\begin{equation}
\mathfrak g = \mathfrak m \oplus \mathfrak a \oplus \mathfrak n.
\end{equation}
Let $G,M,A,N$ be Lie groups of the Lie algebras $\mathfrak g , \mathfrak m ,\mathfrak a ,\mathfrak n$. Suppose that
\begin{itemize}
\item the Lie algebra $\mathfrak n$ is nilpotent and the subgroup $N$ is simply connected (therefore $\exp\colon\mathfrak n\to N$ is a diffeomorphism);
\item the subgroup $MA$ normalizes $N$ and centralizes $A$;
\item there exists an isomorphism $\alpha\colon\mathfrak a\to \mathbf R$, a basis $(e_1,\dots,e_n)$ of $\mathfrak n$ and constants $d_i\geq 1$ such that for any $a\in\mathfrak a$, $\mathop{[}a,e_i\mathop{]}=\alpha(a)d_i$;
\item the group $M$ is compact and is orthogonal for the Euclidean structure $\sum x_i^2$ of $\mathfrak n$ where $(x_1,\dots,x_n)$ is the coordinate system associated to the previous basis $(e_1,\dots,e_n)$;
\item the adjoint representation $\Ad_A$ is injective and $\exp\colon\mathfrak a\to A$ is surjective.
\end{itemize}

The fact that $\mathfrak m \oplus \mathfrak a$ normalizes $\mathfrak n$ implies that the adjoint representation $\ad_G$ defines a restricted map $\ad_{MA}\colon \mathfrak m\oplus \mathfrak a \to \mathfrak{gl}(\mathfrak n)$.
Therefore we see $MA$ as acting on $N$ by conjugation and $N$ acting on itselt by left translation. This action of $G$ on $N$ gives a \emph{nilpotent similarity structure} $(G,\cN)$, where $\cN=N$ is denoted differently to emphasize the fact that the Lie group is thought as a space and not only as a group. We also denote $G$ by $\Sim(\cN)$ when this does not cause any ambiguity.

We can compute  the action of $A$ on $N$ more explicitly. Let $\alpha\colon\mathfrak a\to\mathbf{R}$ be an isomorphism and $(e_1,\dots,e_n)$ be a basis of $\mathfrak n$. By hypothesis for any $a\in\mathfrak a$, we have $[a,e_i] = \alpha(a)d_i$, with $d_i\geq 1$. Therefore
\begin{equation}
\ad_{A}=
\begin{pmatrix}
\alpha d_1 \\ &\alpha d_2 \\ &&\ddots \\ &&&\alpha d_n
\end{pmatrix}
\end{equation}
with zeroes in the blanks.
The equality $\exp(\ad_A(a))=\Ad_{\exp_A(a)}$ gives
\begin{equation}
\Ad_{\exp_A}=
\begin{pmatrix}
\exp(\alpha)^{d_1} \\ &\exp(\alpha)^{d_2} \\ &&\ddots \\ &&&\exp(\alpha)^{d_n}\label{eq-dilatationmatrix}
\end{pmatrix}
.
\end{equation}
By changing $\exp(\alpha)\in\mathbf{R}_+$ for $t\in \mathbf{R}_+$, we get  $A$ equal to $\{\delta_t\}_{t\in \mathbf{R}_+}$ such that $\delta_t x_i = t^{d_i}x_i$ for any $x_i\in \mathbf{R}e_i$. Such a group is called a \emph{dilatation group}.

\paragraph{Example 1}\label{ex-1-iwasawa}
For each semisimple Lie group, there is an Iwasawa decomposition $KAN$ (see \cite{Knapp}) such that $K$ is compact, $A$ is abelian and $N$ is nilpotent. Take $\mathfrak k , \mathfrak a$ and $\mathfrak n$ the corresponding Lie algebras of $K$, $A$ and $N$. Take $\mathfrak m$ the centralizer of $\mathfrak a$ in $\mathfrak k$. Now, reduce $\mathfrak a$ such that it has real dimension one. Then this gives a nilpotent similarity structure. (See \cite[proposition 6.40]{Knapp}.)

In particular, the boundary geometries of the different hyperbolic spaces (with base field $\mathbf F$ that can be the field of real or complex or quaternionic or octonionic numbers) -- that is $(G,X)$ structures with $G = {\rm PU}_\mathbf{F}(n,1)$ and $X=\partial \mathbf H_\mathbf{F}^n$ -- give  the subgeometries of a stabilized point $({\rm PU}_\mathbf{F}(n,1)_p, \partial \mathbf H_\mathbf{F}^n-\{p\})$ as nilpotent similarity structures.

\paragraph{Example 2}
Let $N$ be a Carnot group  \cite{Pansu}. That is to say, let $\mathfrak n = \mathfrak n_1\oplus \dots\oplus \mathfrak n_r$ be a nilpotent graded Lie algebra such that $[\mathfrak n_1,\mathfrak n_i]=\mathfrak n_{i+1}$. A group $N$ is a Carnot group if it corresponds to such a Lie algebra $\mathfrak n$ and is simply connected. A dilatation group $A$ acting by automorphisms on $N$ is naturally given by $\delta_t x_1 = tx_1$ for $x_1\in \mathfrak n_1$. For any other $x_i\in \mathfrak n_i$ we have by construction $\delta_t x_i = t^{i}x$. In particular, it is a dilatation group where $d_i$ is always an integer. For example $H$-type groups (e.g. the group of Heisenberg) are Carnot groups.

\paragraph{Example 3}Damek and Ricci  \cite{Damek1,Damek2} defined a class of harmonic Riemannian spaces that may not be coming from symmetric spaces. If $N$ is a $H$-type group, then in particular $N$ is a two-step nilpotent Lie group equipped with an inner product. The Lie algebra of $N$ naturally decomposes itself into $\mathfrak n=\mathfrak v\oplus \mathfrak z$, with $\mathfrak z$ being the center of $\mathfrak n$ and $\mathfrak v$ the orthogonal of $\mathfrak z$. Then (following \cite{Damek1}) one can take the Lie algebra  $\mathfrak n \oplus \R T$, where $T$ is the transformation given by
\begin{equation}
\ad(T)(X+Z) = [T,X+Z]= \frac 12 X + Z,
\end{equation}
with $X\in\mathfrak v$ and $Z\in\mathfrak z$. Let $S=NA$ be the corresponding simply connected Lie group extension of $\mathfrak n\oplus\R T$. By denoting $A(s)=\exp_S(sT)$,  $A(s)$ acts on $N$ by $A(s)(z,x) = (\sqrt sz,sx)$. By substituting $s=t^2$ (we could also have taken $2T$ instead of $T$), we get $A(t)(z,x)=(tz,t^2x)$. This defines a nilpotent similarity structure $(S,N)$.

\paragraph{}
Consider $(\Sim(\cN),\cN)$ a nilpotent similarity structure.
If $\omega_\cN$ denotes the Maurer-Cartan form on $\cN$, then we can define $\exp_x(tu)$ as the integral line of the left-invariant vector field $\omega_\cN X = u$. Now, by construction $MA$ leaves invariant the set of $N$-left-invariant vector fields on $\cN$. Indeed, denote by $\rho$ the conjugation by $P\in MA$: $\rho(x)=PxP^{-1}$ for $x\in \cN$. Also, let $X= (L_x)_* u$ be the left-invariant vector field such that $\omega_\cN X= u\in\mathfrak n$.
\begin{align}
\rho L_x &= L_{\rho(x)}\rho\\
\implies \rho_* \left(L_x\right)_* &= \left(L_{\rho(x)}\right)_* \rho_*\\
\implies \rho_* X &= \rho_* \left(L_x\right)_* u\\
&= \left(L_{\rho(x)}\right)_* \rho_* u
\end{align}
Therefore, the action of $MA$ on the vector field $X$ gives again a left-invariant vector field on $\cN$.

This allows to define a geodesic structure with compatible holonomy, denoted $(\Sim(\cN),\cN,\exp)$ and where $\exp$ is the Lie exponential map $\exp\colon\mathfrak n \to \cN$. Since $N$ is nilpotent, the exponential gives a diffeomorphism from $\mathfrak n$ to $\cN$. Hence, this structure is injective.
Note that geodesics from $0$ are given by $\exp(tu)$, and geodesics from any $x\in \cN$ are given by $x\exp(tu)$.

This geodesic structure differs generaly from the classic Riemannian structure. Take the Heinseberg group. The classic Riemannian structure induced by a contact form  makes Legendrian curves geodesic. In particular those classic geodesics are not ours since we took as geodesics the straight lines $\exp(tu)$ (they remain straight after left-translation by the Campbell-Hausdorff formula).

\begin{lemma}\label{lem-stabcomp}
Closed $(MN,\cN)$-manifolds are complete.
\end{lemma}

By a result of Auslander \cite[th. 1]{Auslander}, if $\Gamma$ is a discrete cocompact subgroup of $MN$, then in fact $\Gamma\cap N$ has finite index in $N$ and is cocompact. Therefore closed $(MN,\cN)$-manifolds are all given (up to finite index) by cocompact subgroups of $N$.

\begin{proof}
It is a classic general fact that if a geometric structure $(G,X)$ has compact stabilizers, then closed $(G,X)$-manifolds are complete (see \cite{ThurstonBook}).
Let $m\in M$ and $n\in N$. If for $x\in \cN$, $m(nx)m^{-1}=x$ then $m(nx)=mx$, hence $nx=x$. It can only be true if $n=e$. Since $M$ is compact, $(MN,\cN)$ has compact stabilizers.
\end{proof}

From now on, we will denote  the group law of $N$ by the addition. In general, this law is not commutative, since $N$ is only nilpotent. When this will be of use, we will distinguish the addition in $\cN$ from the addition in $\mathfrak n$ by  respectively denoting $+_\cN$ and $+_\mathfrak{n}$. Note that since $\exp\colon\mathfrak n\to\cN$ is a diffeomorphism, there is a rule for exchanging $+_\cN$ and $+_\mathfrak{n}$ by looking at the corresponding coordinates in $\mathfrak n$ given by $\ln=\exp^{-1}$. The explicite rule is exactly given by the Campbell-Hausdorff formula (see e.g. \cite{Knapp}). To be more specific (with $H$ the function given by the Campbell-Hausdorff formula):
\begin{align}
\ln(x+_\cN y)&= H(\ln(x),\ln(y)),\\
H(a,b) &= a+_\mathfrak{n}b +_\mathfrak{n} \frac 12 [a,b] +_\mathfrak{n} \frac 12[a,[a,b]] -_\mathfrak{n} \frac 12 [b,[a,b]] +_\mathfrak{n}\dots\label{eq-campbell}
\end{align}

Also, we denote by product the action of $MA$ on $\cN$ by conjugation. For any $f\in \Sim(\cN)$ we define $\lambda_f\in A, P_f\in M, c_f\in \cN$ such that $f(x) = \lambda_fP_f(x) +c_f$.

\begin{lemma}\label{lem-centerform}
Suppose that for $f\in \Sim(\cN)$, there exists $\beta$ such that $f(\beta)=\beta$. Then $f(x) = \lambda_f P_f(x-\beta)+\beta$.
\end{lemma}
\begin{proof}
By hypothesis $f(x) = \lambda_f P_f(x)+c_f$. Now $f(\beta)=\beta=\lambda_fP_f(\beta)+c_f$. This gives $c_f = -\lambda_fP_f(\beta)+ \beta$. Since $MA$ acts by conjugation on $N$ and $A$ commutes with $M$, we get $\lambda_fP_f(x) - \lambda_fP_f(\beta) = \lambda_fP_f(x-\beta)$. It follows that $f(x) = \lambda_fP_f(x)+c_f = \lambda_fP_f(x-\beta)+\beta$.
\end{proof}

The following proposition implies that theorem \ref{thm-nilpotent} is true as soon as the holonomy group is discrete. It is the main result of \cite{Lee}.

\begin{proposition}\label{prop-limdiscret}
Suppose that a subgroup $\Gamma\subset \Sim(\cN)$ is discrete. Let $f\in \Gamma$ and $a\in \cN$ be such that for any $x\in \cN$, we have $f^nx\to a$. Then for any $g\in \Gamma$, $g(a)=a$.
\end{proposition}
\begin{proof}
Let $g\in \Gamma$. Suppose that $g(a)\neq a$.
The map $h=g\circ f \circ g^{-1}$ has $g(a)$ as fixed point. We set $h_n = f^n\circ h\circ f^{-n}$ and it has $f^n(g(a))$ as fixed point.
Write $h_n(x) = L(h_n)(x)+c_n$ with $L(h_n)$ the $MA$ part of $h_n$. We have
\begin{align}
L(h_n) &= L(f^n)L(h)L(f^{-n})\\
&= \left(\lambda_f^nP_f^n\right)\left(\lambda_gP_g\lambda_fP_f\lambda_g^{-1}P_g^{-1} \right)\left(\lambda_f^{-n}P_f^{-n} \right)\\
&= \lambda_f P_f^n P_gP_fP_g^{-1}P_f^{-n}\\
h_n(x)=&=\lambda_f P_f^nP_g P_fP_g^{-1}P_f^{-n}(x) + c_n
\end{align}
with $c_n$ a constant converging to some constant when $n\to \infty$. Indeed, $f^n(g(a))\to a$ by hypothesis, $h_n$ fixes $f^n(g(a))$ and
\begin{align}
c_n &=-\lambda_f P_f^nP_g P_fP_g^{-1}P_f^{-n}\left(f^n(g(a))\right)+ h_n\left(f^n(g(a))\right)\\
&\to -\lambda_f P_f^nP_g P_fP_g^{-1}P_f^{-n}(a) + a
\end{align}
By compacity of the subgroup $M\subset \Sim(\cN)$, the map $P_f^nP_g P_fP_g^{-1}P_f^{-n}$ converges towards some $P\in M$. This shows that $c_n$ converges aswell and therefore $h_n$ converges towards some map $h'\in \Sim(\cN)$. But $(h_n)$ is not a constant sequence of maps by hypothesis on $g$. It raises a contradiction since $\Gamma$ is discrete.
\end{proof}

In general, let $\{\delta_t\}_{t\in\mathbf{R}_+}$ be a one-parameter subgroup of linear automorphisms acting on a nilpotent Lie algebra $\mathfrak n$. Furthermore, assume that there is a coordinate basis $x_1,\dots,x_n$ of $\mathfrak n$ such that $\delta_tx_i=t^{d_i}x_i$ for any $i$ and with $d_i\geq 1$. Those transformations are called \emph{dilatations}. A very general theorem of Hebisch and Sikora~\cite{Hebisch} allows to have a pseudo-norm on $\cN$.

\begin{theorem}[\cite{Hebisch}]
Let $\cN$ be a nilpotent group with dilatations $\{\delta_t\}$. There exists a symetric pseudo-norm $\|\cdot\|_\cN$ (also denoted $\|\cdot\|$ when no confusion is possible) on $\cN$ such that the following properties are true.
\begin{enumerate}
\item $\|x+_\cN y\| \leq \|x\| +\|y\|$;
\item $\|\delta_tx\| = t\|x\|$;
\item $\|x\|=0\iff x=0$;
\item $\|-x\|=\|x\|$;
\item $\|\cdot\|$ is continuous on $\cN$ and smooth on $\cN-\{0\}$;
\item the unit ball $\|x\|<1$ is the Euclidean ball $\sum x_i^2 < r^2$ for some $r<\eps$ small enough which has to be chosen first.
\end{enumerate}\label{thm-Hebisch2}
The last property relies on the second theorem of \cite{Hebisch}, and is not true for any pseudo-norm verifying the other five conditions.
\end{theorem}

With such a pseudo-norm on $\cN$, we can define a distance function and a good notion of open balls.
First, we define the distance function by
\begin{equation}
d_\cN(x,y) \coloneqq \|-x+_\cN y\|_\cN.
\end{equation}
This distance function is left-invariant by the action of $\cN$ on itself.
We define the \emph{open balls} of $\cN$ to be
\begin{equation}
B(x,R) \coloneqq \{y\in\cN \,\mid\, d(x,y)<R\}.
\end{equation}

Note that for any $x\in \cN$, $x+B(0,R)=B(x,R)$. Also if $\rho\in MA$, then we have $\rho B(x,R)=B(\rho(x),\lambda_\rho R)$. Indeed,
\begin{align}
d(\rho(x),\rho(y)) &= \|-\rho(x)+\rho(y)\| \\
&= \|\rho(-x+y)\|\\
&=\lambda_\rho \|-x+y\| = \lambda_\rho d(x,y).
\end{align}
Note that we used the fact that $M$ preserves the Euclidean metric $\sum x_i^2$: it therefore preserves the unit ball centered in $0$ and hence any ball centered in $0$.

\begin{proposition}\label{prop-ballconvex}
The open balls $B(x,R)$ are geodesically convex (in the sense of definition \ref{def-geostruct}).
\end{proposition}
\begin{proof}
First, we examine the fact that any two points in an open ball have a geodesic segment that is contained in the open ball.
By left-invariance, it is only required to prove this result for $B(0,R)$.

We first show that this ball is convex in the usual sense. That is to say, we show that if $x,y\in B(0,R)$ then the linear combination $tx+_\mathfrak{n}(1-t)y$ also belongs to $B(0,R)$ for $t\in \iI$ and for $+_\mathfrak{n}$ the linear addition in the coordinates given in $\mathfrak n$ by the inverse of $\exp\colon\mathfrak n \to \cN$.
By the second theorem of \cite{Hebisch}, we can suppose that the unit ball is Euclidean and of radius $r$ as small as desired (but fixed), hence convex for linear segments. Now take $R>0$, and suppose that $x,y\in B(0,R)$. To show that the linear segment $\gamma\colon\iI \to \cN$ from $x$ to $y$ is contained in $B(0,R)$, we show that $\delta_{1/R}\gamma$ is contained in $B(0,1)$. By linearity of the action of $\delta_t$ on $\mathfrak n$,
\begin{equation}
\delta_{1/R}(tx+_{\mathfrak n}(1-t)y) = t\delta_{1/R}x +_{\mathfrak n} (1-t)\delta_{1/R}y
\end{equation}
and since $\delta_{1/R}x,\delta_{1/R}y$ are both in $B(0,1)$ and the last expression is a linear segment with extremities in $B(0,1)$, the segment $\delta_{1/R}\gamma$ is fully contained in $B(0,1)$.

Now we prove that $B(0,R)$ contains its geodesic segments for any $R>0$. We make the following observation: geodesics from $0$ are given by $\exp(tu)$, so the coordinates are $tu$. By the Campbell-Hausdorff formula, coordinates of geodesics now based in $\exp(x)$ are given by the equation
\begin{equation}
\exp(x)\exp(tu) = \exp\left(x+_\mathfrak{n}tu+_\mathfrak{n} \frac12[x,tu] +_\mathfrak{n} \frac 1{12}[x,[x,tu]] -_\mathfrak{n} \frac 1{12}[y,[x,y]] +_\mathfrak{n} \dots \right)
\end{equation}
but when $x$ is close to $0$, and $tu$ is close to $0$ then
\begin{equation}
\exp(hx)\exp(htu) = \exp\left(hx+_\mathfrak{n}htu +_\mathfrak{n} o\left(h^2\right)\right).
\end{equation}
Therefore, essentially, geodesics close to $0$ are linear segments.

If $x,y\in B(0,R)$, then up to applying $\delta_t$ for $t$ small enough, we can assume $tR=\eps$ and $\|\delta_tx\|,\|\delta_ty\| < \eps$. In fact, the entire ball is transformed: $\delta_t(B(0,R))=B(0,tR)$. Let $\gamma\colon\iI\to \cN$ be the geodesic segment from $x$ to $y$. Then $\delta_t\gamma$ is the geodesic segment from $\delta_tx$ to $\delta_t y$. The fact that $\gamma$ is entirely contained in $B(0,R)$ is equivalent to the fact that $\delta_t\gamma$ is entirely contained in $B(0,\eps)$. Since both $\gamma$ and $\partial B(0,R)$ are given by smooth functions, assume that for $t_1,t_2$ we have, for any $s\in \mathop{]}t_1,t_2\mathop{[}$, $\gamma(s)\not\in B(0,R)$ and $\gamma(t_1),\gamma(t_2)\in\partial B(0,R)$. By injectivity of the geodesic structure, $\gamma(s)$ is the only geodesic segment from $\gamma(t_1)$ to $\gamma(t_2)$. Now $\delta_t\gamma(s)$ does never belong to $B(0,\eps)$ for any $\eps>0$. On the other hand, $\delta_t\gamma(s)$ tends to a linear-segment by the preceding observation, hence it must at least intersect $B(0,\eps)$, since it contains the linear-segment from $\delta_t\gamma(t_1)$ to $\delta_t\gamma(t_2)$. Absurd, therefore $\gamma$ is entirely contained in $B(0,R)$. This finishes to prove that any two points in an open ball give a geodesic segment that is contained in this ball.

\paragraph{}
We now prove the second requirement for convexity (see definition \ref{def-geostruct}). To be more specific, let $\gamma_n$ be a sequence of geodesic segments with a fixed base point $p\in B(x,R)$ and with endpoints $q_n$ converging towards $q\in \overline{B(x,R)}$. We need to show that $\gamma_n$ has in fact a subsequence converging to $\gamma$, which is a geodesic segment from $p$ to $q$.

Again by left invariance, it suffices to show this fact for $p=0\in \cN$. With this assumption, each $\gamma_n$ is given by $\gamma_n(t)=\exp(tv_n)$ with $v_n\in\mathfrak n$. By assumption $\exp(v_n)\to q$. Denote $q=\exp(v)$. Since the exponential map is a diffeomorphism, this implies  that $v_n\to v$. Now it is clear that the sequence $\exp(tv_n)$ converges towards $\exp(tv)$, and this is precisely the geodesic from $p$ to $q$.
\end{proof}

\paragraph{Open questions for nilpotent similarity structures in higher ranks}We did suppose that $A$ is a rank one group of dilatations. But for example if we consider the Iwasawa decomposition of a semisimple Lie group, then the dilatation group $A$ is of the same rank as the semisimple group, and in particular has no reason to be of rank one. So now, until the end of this section, we consider $A$ a group of dilatations with a higher rank. We no longer have a nice symmetric pseudo-norm as we had before, and therefore no longer a nice sense of open balls that would remain invariant under the transformation group. (An open ball is still convex, but is not necessarily transformed into another open ball.)
\begin{enumerate}
\item If $M$ is a closed higher rank manifold, when does $M$ have a holonomy group with a dilatation subgroup of rank more that one?
\item Is there a nice general class of geometrical objects that are convex, open, and invariant under the transformation group?
\item When is theorem \ref{thm-nilpotent}  true in higher rank?
\end{enumerate}
The answer to the first question  would generalize the conclusion of theorem \ref{thm-nilpotent} for manifolds with a higher rank nilpotent similarity structure.
For the second question, it should be pointed out that the work of Carrière~\cite{Carriere} about Lorentzian manifolds relies on the fact that the set of the ellipsoids remains stable under the transformation group (and allows Carrière to study the “discompacity” of a holonomy group).

We now discuss the third question.
First, note that proposition \ref{prop-limdiscret} remains true, since the proof did not rely on the fact that $A$ was of rank one. But the existence of $f$ such that $f^n(x)\to a$ depends on a more specific description of $A$. By lemma \ref{lem-stabcomp}, an incomplete structure on a closed manifold implies the existence of $f$ such that the $A$ part of $f$ is not the identity map, say $A_f$. If $A_f^n$ or $A_f^{-n}$ tends to the null transformation, then we can apply proposition \ref{prop-limdiscret}. Say this is the case. In order to have a (very weaker) discrete version of theorem \ref{thm-nilpotent}, we still need to have a version of lemma \ref{lem-stabcomp} for $(MA,\cN)$-closed manifolds. In higher rank, there is no evident reason for this to be true.

\paragraph{A counterexample in rank 2}This example is  part of a paper of Aristide~\cite[p. 3699]{Atistide}. This paper is about radiant manifolds with a similarity group looking like $\SO(n,1)\mathbf R_+$. (We could think $\SO(n,1)$ as a group somehow constituted of a subgroup of $O(n+1)$ and of dilatations.)

Consider $\mathbf R^2$ and the “pseudo-similarity” group $G$ constituted of translations and of a rank $2$ dilatation group, generated by $\delta_t(x,y) = (tx,ty)$ and $\rho_s(x,y)=(sx,s^2y)$. In this dilatation group we consider $\lambda_t$ defined by $\lambda_t=\rho_t\delta_{t}^{-1}$, thus giving $\lambda_t(x,y)=(x,ty)$.

Now consider the subgroup $\Gamma$ generated by $\gamma_1(x,y)=(x+1,y)$ and $\gamma_2(x,y)=(x,2y)$. We have $\Gamma\subset G$.
The quotient $\mathbf R \times\mathbf R^*_+/\Gamma$ is compact, and diffeomorphic to a torus.
Of course, this is not a complete affine structure on the torus. Since $\Gamma$ does not fixe any point of $\mathbf R^2$, the conclusion of theorem \ref{thm-nilpotent} does not hold, providing a counterexample when the rank of $A$ is no longer $1$.

Note that this example is particular in the sense that the quotient is in fact a product of a Euclidean manifold and a radiant manifold.

\paragraph{Open question}What are the affine manifolds which are diffeomorphic to a product $E_1\times\cdots\times E_k\times R_1\cdots R_m$ where the $E_i$ are Euclidean (or nilpotent) manifolds and $R_j$ are radiant manifolds?

Those manifolds (with $m\neq 0$) are incomplete affine manifolds, and still comply to  the Chern conjecture.

\section{Closed nilpotent similarity manifolds}\label{sec-4}

Throughout this section, let $\Sim(\cN)$ denote a similarity group acting on a nilpotent space $\cN$ as defined in the preceding section.
For any $f\in \Sim(\cN)$ we define $\lambda(f)\in A, P(f)\in M, c(f)\in \cN$ such that $f(x) = \lambda(f)P(f)(x) +c(f)$. Also, since this creates no ambiguity, for $g\in \pi_1(M)$ we denote $\lambda(g)$, $P(g)$, $c(g)$ instead of $\lambda(\rho(g))$, $P(\rho(g))$ and $c(\rho(g))$.

\begin{theorem}\label{thm-nilpotent}
Let $M$ be a connected closed $(\Sim(\cN),\cN)$-manifold. If the developing map $D\colon\widetilde M \to \cN$ is not a diffeomorphism, then the holonomy group $\Gamma=\rho(\pi_1(M))$ fixes a point in $\cN$ and $D$ is in fact a covering onto the complement of this point.
\end{theorem}

\emph{The proof of this result begins here and ends at the end of the section.}

Corollary \ref{cor-convexity} and proposition \ref{prop-convgeocomp} show that for every $p\in \widetilde M$, the set $V_p\neq \rT_p\widetilde M$. Hence the image $D(\exp_p(V_p)) = \exp_{D(p)}(\dd D_p(V_p))$ is never equal to $\cN$ by the injectivity of the geodesic structure $(\cN,\exp)$.
Theorefore for each $p\in \widetilde M$, there exists an open subset $B_p\subset V_p\subset {\rm T}_p\widetilde M$ such that the image $D(\exp_p(B_p))$ is the maximal open ball in $D(\exp_p(V_p))$ centered in $D(p)$.
We let
\begin{equation}
 r\colon \widetilde M \to ]0,+\infty[
\end{equation}
be the map that associates the radius of the ball $D(\exp_p(B_p))$ in $\mathcal{N}$ to $p\in \widetilde M$. Note that since the structure is injective and since open balls are geodesically convex by proposition \ref{prop-ballconvex}, the set $\exp_p(B_p)$ is always convex for any $p\in \widetilde M$.

\begin{lemma}\label{lem-rcontract}
For $p\in \widetilde M$ and $q\in \exp_p(B_p)$,
\begin{equation}
r(p)\leq r(q)+d_{\mathcal N}(D(p),D(q)).
\end{equation}
Furthermore, if $g\in\pi_1(M)$ then $r(gp)=\lambda(g)r(p)$ with $\lambda(g)$ being the dilatation factor of the holonomy transformation $\rho(g)\in{\rm Sim}(\mathcal{N})$.
\end{lemma}

\begin{proof}
Let $p\in \widetilde M$
and let $q\in \exp_p(B_p)$.
By lemma \ref{lem-convvisible}, $\exp_p(B_p)\subset \exp_q(V_q)$.
Let $v\in \partial B_q$ such that $\exp_q(v)$ is not defined. Then $\exp_{D(q)}(\dd D_q (v))$ does not belong to $D(\exp_q(V_q))$ and hence does not belong to $D(\exp_p(B_p))$ either, which is precisely an open ball of radius $r(p)$. Therefore
\begin{align}
r(p) &\leq d_{\mathcal N}\left(D(p),\exp_{D(q)}\left({\rm d}D_q(v)\right)\right) \\
&\leq d_{\mathcal N}(D(p),D(q))+ d_{\mathcal N}\left(D(q),\exp_{D(q)}\left({\rm d}D_q(v)\right)\right) \\
&\leq d_{\mathcal N}(D(p),D(q)) + r(q).
\end{align}
This proves the inequality.

For the second part, we prove that $\rho(g)$ transforms the ball $D(\exp_p(B_p))$ into the ball $D(\exp_{gp}(B_{gp}))$. If that is true then for any $v\in\partial D(\exp_p(B_p))$, we have $\rho(g)v\in \partial D(\exp_{gp}(B_{gp}))$ and therefore
\begin{align}
r(gp) = d_{\mathcal N}(D(gp),\rho(g)v) &= d_{\mathcal N}(\rho(g)D(p),\rho(g)v) \\
&=\lambda(g)d_{\mathcal N}(D(p),v)=\lambda(g)r(p).
\end{align}

In fact, it suffices to prove that $\rho(g) D(\exp_p(B_p))\subset D(\exp_{gp}(B_{gp}))$ since with $g^{-1}$ we would get the other inclusion.
By proposition \ref{prop-convhol}, $g$ sends $\exp_p(B_p)$ into a convex subset containing $gp$, and by lemma \ref{lem-convvisible} this convex is contained in $\exp_{gp}(V_{gp})$. But $\rho(g)$ preserves open balls, hence $\rho(g)D(\exp_p(B_p))$ is an open ball contained in $D(\exp_{gp}(B_{gp}))$ by maximality of $B_{gp}$.
\end{proof}

The equivariance between $r$ and $\lambda$ allows a sense of length in $\widetilde M$ which will be invariant by the holonomy group. By comparing with proposition \ref{prop-trivconv}, we will give a system of trivializing neighborhoods of $M$ for the covering $\pi\colon\widetilde M\to M$ such that those neighborhoods are comparable to the sets $\exp_p(B_p)$.

In $\widetilde M$ we set the pseudo-distance function
\begin{equation}
d_{\widetilde M}\left(p_1,p_2\right) = \frac{d_\mathcal{N}\left(D\left(p_1\right),D\left(p_2\right)\right)}{r\left(p_1\right)+r\left(p_2\right)}
\end{equation}
which is $\pi_1(M)$-invariant by the preceding lemma.
Let $p\in \widetilde M$ and let $\eps>0$. We will describe open balls of radius $\eps$ in $\widetilde M$ by looking locally at the pseudo-distance function $d_{\widetilde M}$ on couples $(p,x)$ with $x\in \exp_p(B_p)$. On those couples, $D$ is injective. Also, by lemma \ref{lem-rcontract},
\begin{align}
d_{\widetilde M}(p,x) &= \frac{d_{\mathcal N}(D(p),D(x))}{r(p)+r(x)} \\
\implies d_{\widetilde M}(p,x) &\geq \frac{d_{\mathcal N}(D(p),D(x))}{2r(x) + d_{\mathcal N}(D(p),D(x))}
\end{align}
hence if we suppose $d_{\widetilde M}(p,x)<\eps$ with $\eps$ small enough, we get
\begin{align}
\frac{d_{\mathcal N}(D(p),D(x))}{2r(x) + d_{\mathcal N}(D(p),D(x))} &< \eps \\
\iff \frac{d_{\mathcal N}(D(p),D(x))}{r(x)} &< \frac{2\eps}{1-\eps} .\label{eq-pseudounif}
\end{align}
If $r(p)\geq r(x)$ then the same inequality is true for $r(p)$ instead of $r(x)$. If $r(p)< r(x)$, then $p\in \exp_x(B_x)$ by lemma \ref{lem-3geodesic}. By repeating the same argument for $(x,p)$ we get the preceding inequality with $r(p)$ instead of $r(x)$. In either cases
\begin{equation}
\frac{d_{\mathcal N}(D(p),D(x))}{r(p)} < \frac{2\eps}{1-\eps} .
\end{equation}

Conversely by using $r(x)\geq r(p)-d_{\mathcal N}(D(p),D(x))$:
\begin{equation}
d_{\widetilde M}(p,x) \leq \frac{d_{\mathcal N}(D(p),D(x))}{2r(p) - d_{\mathcal N}(D(p),D(x))},
\end{equation}
hence for $\eps>0$
\begin{align}
\frac{d_{\mathcal N}(D(p),D(x))}{2r(p) - d_{\mathcal N}(D(p),D(x))} &< \eps \\
\iff \frac{d_{\mathcal N}(D(p),D(x))}{r(p)} &< \frac{2\eps}{1+\eps}
\end{align}

This shows that for $\eps$ small enough, the ball
\begin{equation}
B_{\widetilde M}(p,\eps) \coloneqq \left\{x\in \exp_p\left(B_p\right) \,\mid\, d_{\widetilde M}(p,x)<\eps\right\}
\end{equation}
has an approximation in its developing image:
\begin{equation}
\left\{ \frac{d_{\mathcal N}(D(p),D(x))}{r(p)} < \frac{2\eps}{1+\eps} \right\} \subset D\left(B_{\widetilde M}(p,\eps)\right) \subset \left\{ \frac{d_{\mathcal N}(D(p),D(x))}{r(p)}<\frac{2\eps}{1-\eps} \right\}
\end{equation}
and is contained in an open of $\widetilde M$. Those open balls hence provide the same basis for the topology of $\widetilde M$.
This means that $d_{\widetilde M}(p,x)<\eps$ is true when in the ball $D(\exp_p(B_p))$ normalized by the radius $r(p)$, the distance between $D(x)$ and $D(p)$ is less than $\simeq 2\eps$.

If $g\in \pi_1(M)$, then by the proof of lemma \ref{lem-rcontract}, $gB_{\widetilde M}(p,\eps)\subset g \exp_p(B_p)$ is a subset of $\exp_{gp}(B_{gp})$ since $g\exp_p(B_p)=\exp_{gp}(B_{gp})$. The pseudo-distance function $d_{\widetilde M}$ being $\pi_1(M)$-invariant, this shows that
\begin{equation}
\forall g\in\pi_1(M), \; gB_{\widetilde M}(p,\eps)= B_{\widetilde M}(gp,\eps).
\end{equation}

In $M$, we can define a system of open neighborhoods by projecting the previously constructed balls of $\widetilde M$:
\begin{equation}
B_M(x,\eps) \coloneqq \pi(B_{\widetilde M}(p,\eps)), p\in \pi^{-1}(x).
\end{equation}
For $\eps$ small enough, the ball $B_M(x,\eps)$ is therefore a trivializing neighborhood of $x$, and this system of open balls gives the same topology for $M$ as the original one.

\paragraph{}
We will now construct holonomy transformations which will be very contracting, with a common center point and with no rotation. The idea is to take $v\in\partial B_p$ such that $\exp_p(v)$ is not defined and to compare with $\exp_{D(p)}({\rm d}D (v))$ in $\mathcal{N}$ where it must be defined. The holonomy transformations will be centered in $\exp_{D(p)}({\rm d}D(v))= \exp_z(w)$. See figure \ref{fig} for the global setting.

\begin{figure}[ht]
\centering
\includegraphics[width = \textwidth]{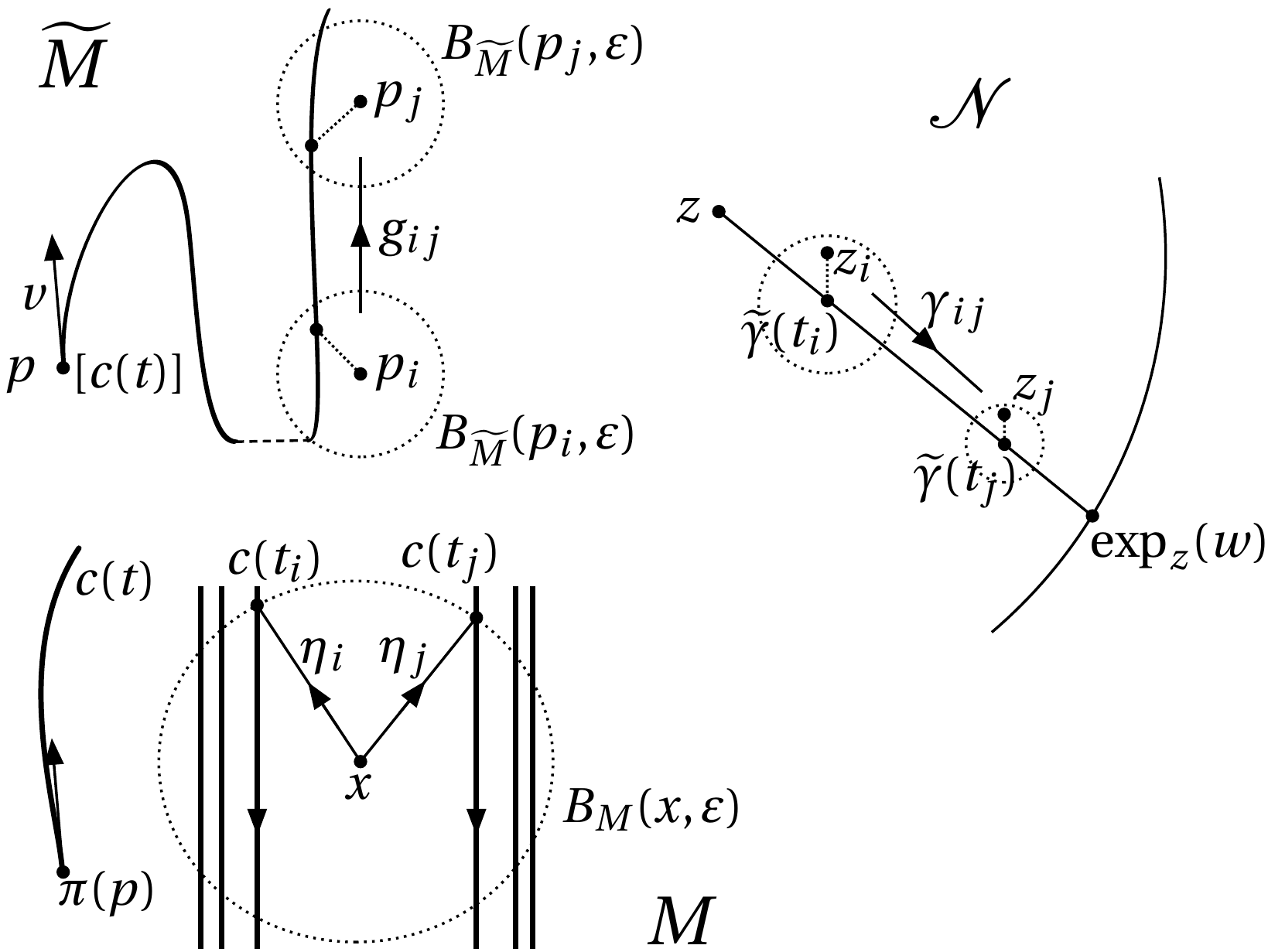}
\caption{The general setting.}\label{fig}
\end{figure}

Consider $p\in \widetilde M$ such that $\exp_p(tv)$ is defined for $0\leq t<1$ but not for $t=1$. The geodesic curve $[c(t)]=\exp_p(tv)$ is an incomplete geodesic. In $M$, the corresponding curve $c(t)=\pi([c(t)])$ is then an infinite long curve in a compact space. There is therefore a recurrent point $x\in M$.

Let $B_M(x,\eps)$ be a ball with radius $\eps>0$ small enough such that $B_M(x,\eps)$ is trivializing $\pi\colon\widetilde M \to M$. Let $0<t_1<\dots<t_n<\dots$ be entry times such that $t_n\to 1$; $c(t_n)\in B_M(x,\eps)$; but $c([t_n,t_{n+1}])\not\subset B_M(x,\eps)$ (it just states that $c$ exits $B_M(x,\eps)$ before time $t_{n+1}$).
Since $\eps$ is small enough, for each $t_n$, up to homotopy we can uniquely set  $\eta_n$ the segment from $x$ to $c(t_n)$ contained in $B_M(x,\eps)$.
By construction we have the following lemma.
\begin{lemma}\label{lem-technique}
For any $i$,
 $[c(t_i)]\in \exp_{p_i}(B_{p_i})$ and $d_{\widetilde M}([c(t_i)],p_i)<\eps$.
\end{lemma}
\begin{proof}
By hypothesis and according to the preceding discussion, since $B_M(x,\eps)$ is trivializing, if $c(t_i)\in B_M(x,\eps)$ then any lift of $c(t_i)$ is in $B_{\widetilde M}(\widehat p,\eps)$ for $\widehat p\in\pi^{-1}(x)$. In particular $[c(t_i)]$ is at a maximum distance $\eps$ from $p_i$ and lies in $\exp_{p_i}(B_{p_i})$ by definition of $B_{\widetilde M}(p_i,\eps)$.
\end{proof}

We define the concatenated path $g_{ij}$ by
\begin{equation}
 g_{ij} = \eta_j^{-1} \star c(t)|_{\left[t_i,t_{j}\right]} \star \eta_i.
\end{equation}
This is a family of transformations belonging to $\pi_1(M,x)$.
We are now interested in $g_{ij}$ acting on $\widetilde M$.
The path $\widetilde g_{ij}$ lifting $g_{ij}$ sends $p_i$ to $p_j$ by construction. We denote by $\gamma_{ij}$ the holonomy transformation $\rho(g_{ij})$ and we denote by $\widetilde \gamma$ the image $D([c(t)])$. Even if $\gamma(1)$ is not defined, $\widetilde\gamma(1)$ is well defined.

The initially chosen vector $v$  is sent to $w$ by ${\rm d}D$ and $z\coloneqq D(p)$. Each $p_i$ is sent to $z_i$ by $D$. See again figure \ref{fig}.

\begin{proposition}
Suppose that $j\gg i \to \infty$ and we are up to choose a subsequence of $(i,j)$. The transformation $\gamma_{ij}$ verifies the following properties.
\begin{enumerate}
\item $P(\gamma_{ij}) \to E$ (With $E$ the identity transformation.)
\item $\lambda(\gamma_{ij})\to 0$
\item For any $x\in \cN$, $\gamma_{ij}(x)\to \exp_z(w)=\widetilde\gamma(1)$.
\end{enumerate}
\end{proposition}

\begin{proof}
Since $M$ is compact and since $\gamma_{jk}\gamma_{ij}=\gamma_{ik}$, the transformation $P(\gamma_{ij})$ accumulates to the identity $E$. This proves the first point up to choose a subsequence of $(i,j)$.

Now we prove that $\lambda\to 0$.
By lemmas \ref{lem-technique} and \ref{lem-rcontract}
\begin{equation}
r(p_j) \leq r([c(t_j)]) + d_{\mathcal N}(z_j,\widetilde\gamma(t_j))
\end{equation}
and by lemma \ref{lem-technique} and the definition of $d_{\widetilde M}$
\begin{equation}\label{eq-approx}
d_\mathcal{N}(z_j,\widetilde\gamma(t_j)) < \eps \left( r(p_j) + r([c(t_j)]) \right).
\end{equation}
This shows that for $j\to \infty$, $d_\cN(z_j,\widetilde\gamma(t_j))$ is arbitrarily small.
This also gives
\begin{equation}
r(p_j) \leq (1+\eps) r([c(t_j)]) + \eps r(p_j).
\end{equation}
Note $r(p_j) = r(g_{ij}(p_i)) = \lambda(g_{ij}) r(p_i)$, it follows that
\begin{align}
\lambda(g_{ij}) r(p_i) &\leq (1+\eps) r([c(t_j)]) + \eps\lambda(g_{ij}) r(p_i)\label{eq-53}\\
\iff \lambda(g_{ij}) &\leq \frac{1+\eps}{1-\eps} \frac{r([c(t_j)])}{r(p_i)}.\label{eq-lambda}
\end{align}
Since the numerator tends to $0$, for $i$ fixed we get that $\lambda\to 0$ for $j$ tending to $+\infty$. So this is true for $i,j$ large enough such that $j\gg i$. This proves the second point.

The preceding calculation showed that $d_\cN(z_j,\widetilde\gamma(t_j))$ is arbitrarily small, and when $t_j\to 1$ we have $\widetilde\gamma(t_j)\to \exp_z(w)=\widetilde\gamma(1)$. Therefore $z_j\to \exp_z(w)$.
By construction, $\gamma_{ij}$ sends $D(p_i)=z_i$ to $D(p_j)=z_j$. Now:
\begin{align}
d_\cN(\gamma_{ij}(z_i),\gamma_{ij}(x)) &= \lambda(\gamma_{ij})d_\cN(z_i,x)
\end{align}
and since $\lambda(\gamma_{ij})\to 0$, we have $\gamma_{ij}(x)\to \gamma_{ij}(z_i)=z_j\to \exp_z(w)$. This proves the third and last point.
\end{proof}

\begin{lemma}\label{lem-approxgeo}
We have the following inequalities
\begin{align}
d_\cN(\rho(g_{ij})\widetilde\gamma(t_i),\widetilde\gamma(t_j)) &\leq r([c(t_j)]) \frac{4\eps-4\eps^2}{1-5\eps+3\eps^2},\\
\lambda(g_{ij})&\leq \frac{1+\eps}{1-3\eps}\frac{r([c(t_j)])}{r([c(t_i)])}.
\end{align}
\end{lemma}
\begin{proof}
We estimate first $r([c(t_j)])$ in terms of $r([c(t_i)])$ and $\lambda(g_{ij})$. By the preceding calculations, we know that
\begin{equation}
\lambda(g_{ij})r([c(t_i)]) \leq \frac{1+\eps}{1-\eps}\frac{r([c(t_j)])}{r(p_i)}r([c(t_i)]).
\end{equation}
But $[c(t_i)]\in \exp_{p_i}(B_{p_i})$ and $d_{\widetilde{M}}([c(t_i)],p_i)<\eps$. Therefore by equation \eqref{eq-pseudounif}:
\begin{equation}
d_{\cN}(D(p_i),\widetilde\gamma(t_i)) < \frac{2\eps}{1-\eps} r([c(t_i)]).
\end{equation}
Since we can choose $\eps$ small enough such that $2\eps/(1-\eps)<1$, we see that in fact 
\begin{equation}
d_\cN(D(p_i),\widetilde\gamma(t_i))<r([c(t_i)]).
\end{equation}
We therefore have $p_i\in \exp_{[c(t_i)]}(B_{[c(t_i)]})$ (see lemma \ref{lem-3geodesic}). It follows by lemma \ref{lem-rcontract} and the previous inequality that
\begin{align}
r([c(t_i)]) &\leq r(p_i) + d_\cN(\widetilde\gamma(t_i),D(p_i))\\
&\leq r(p_i) + \frac{2\eps}{1-\eps} r([c(t_i)]) \\
\implies \frac{1-3\eps}{1-\eps}r([c(t_i)]) &\leq r(p_i).
\end{align}
Therefore
\begin{align}
\lambda(g_{ij})r([c(t_i)]) & \leq \frac{1+\eps}{1-\eps}\frac{r([c(t_j)])}{r(p_i)}r([c(t_i)])\\
&\leq \frac{1+\eps}{1-\eps} r([c(t_j)]) \frac{1-\eps}{1-3\eps}\\
&\leq \frac{1+\eps}{1-3\eps}r([c(t_j)]).
\end{align}
This proves the second inequality.

Now, we can compute the distance between $\rho(g_{ij})\widetilde\gamma(t_i)$ and $\widetilde\gamma(t_j)$. Recall equation \eqref{eq-53}.
\begin{align}
d_\cN(\rho(g_{ij})\widetilde\gamma(t_i),\widetilde\gamma(t_j)) &\leq d_\cN(\rho(g_{ij})\widetilde\gamma(t_i),\rho(g_{ij})D(p_i)) + d_\cN(\rho(g_{ij})D(p_i),\widetilde\gamma(t_j))\\
&\leq \lambda(g_{ij})d_\cN(\widetilde\gamma(t_i),D(p_i)) + d_\cN(D(p_j),\widetilde\gamma(t_j))\\
&\leq \lambda(g_{ij})\left(\eps (r(p_i)+r([c(t_i)]))\right) + \eps(r(p_j)+r([c(t_j)]))\\
&\leq \eps\left( \lambda(g_{ij})r([c(t_i)]) \left( \frac{1+\eps}{1-\eps} + 1 \right) + r([c(t_j)])\left( \frac{1+\eps}{1-\eps} + 1 \right)\right)\\
&\leq \eps r([c(t_j)])\left( \frac{1+\eps}{1-3\eps}\frac{2}{1-\eps} + \frac{2}{1-\eps} \right) \\
&\leq \eps r([c(t_j)])\frac{2+2\eps + 2 -6\eps}{1-5\eps +3\eps^2}\\
&\leq \eps r([c(t_j)]) \frac{4-4\eps}{1-5\eps+3\eps^2}
\end{align}
It gives the required inequality.
\end{proof}

Note that those inequalities allow to compute a bit more explicitly how $\rho(g_{ij})$ moves $\widetilde\gamma(1)= \exp_{z}(w)$. We already know that for $j\gg i$ both large enough, it sends $\widetilde\gamma(1)$ close to itself.
\begin{align}
d_\cN(\rho(g_{ij})\widetilde\gamma(1),\widetilde\gamma(1)) &\leq d_\cN(\rho(g_{ij})\widetilde\gamma(1),\rho(g_{ij})\widetilde\gamma(t_i))+d_\cN(\rho(g_{ij})\widetilde\gamma(t_i),\widetilde\gamma(1))\\
&\leq \lambda(g_{ij})r([c(t_i)]) + d_\cN(\widetilde\gamma(t_j),\rho(g_{ij})\widetilde\gamma(t_i))+d_\cN(\widetilde\gamma(t_j),\widetilde\gamma(1))\\
&\leq r([c(t_j)]) \left( \frac{1+\eps}{1-3\eps} + \frac{4\eps-4\eps^2}{1-5\eps+3\eps^2}+ 1\right)
\end{align}

We now recall the construction made for proposition \ref{prop-convG}.
\emph{Let $p\in \widetilde M$. If for $g\in\pi_1(M)$, $g\cdot p$ is visible from $p$ by a vector in $B_p$, then if a vector $u$ is such that $G(u)\in B_p$, then in fact $\exp_p(u)$ is visible from $p$, with $G$ given by}
\begin{equation}
G = \dd D_{p}^{-1}\circ \exp_{D(p)}^{-1}\circ \rho(g)\circ \exp_{D(p)}\circ \dd D_p.
\end{equation}
This allows to extend the set of vectors that can be seen from $p$ further than $B_p$.
By passing through the developing map, if ${\rm d}D_pG(u)\in {\rm d}D_p(B_p)$ then it is true that $G(u)\in B_p$. Therefore if $\exp_{D(p)}(\dd D_p G(u))\in D(\exp_p(B_p))$ then $u$ is visible.

\begin{lemma}
For $j\gg i$ large enough, the point $g_{ij}p$ belongs to $\exp_p(B_p)$.
\end{lemma}

\begin{proof}Recall lemma \ref{lem-3geodesic}.
To prove that $g_{ij}p$ is visible from $p$, we will show that: $p_j$ belongs to $\exp_p(B_p)$; $g_{ij}p$ is visible from $p_j$; and $D(g_{ij}p)\in D(\exp_p(B_p))$. Those three facts and lemma \ref{lem-3geodesic} show that $g_{ij}p$ belongs to $\exp_p(B_p)$.

To prove the first point, we use  lemma \ref{lem-3geodesic} again. It will be true for in fact any integer $i$. Of course, $[c(t_i)]$ belongs to $\exp_p(B_p)$. By construction (see lemma \ref{lem-technique}), $p_i$ is visible from $[c(t_i)]$.
By lemma \ref{lem-technique}, by definition of the pseudo distance function $d_{\widetilde{M}}$ and by equation \eqref{eq-pseudounif}:
\begin{equation}
d_{\cN}(D(p_i),D([c(t_i)])) < \frac{2\eps}{1-\eps} r([c(t_i)]).
\end{equation}
Also note that $d_\cN(D(p),D([c(t_i)])=r(p)-r([c(t_i)])$, it follows that:
\begin{align}
d_\cN(D(p),D(p_i)) &\leq d_\cN(D(p),D([c(t_i)]) + d_\cN(D([c(t_i)]),D(p_i))\\
&< r(p)-r([c(t_i)]) + \frac{2\eps}{1-\eps}r([c(t_i)]) \\
&< r(p) - \frac{1-3\eps}{1-\eps} r([c(t_i)])<r(p)
\end{align}
And the last step comes from the initial choice of $\eps$ small enough.
This shows that $D(p_i)\in D(\exp_p(B_p))$ and therefore $p_i\in \exp_p(B_p)$.

Now, take a geodesic segment from $p$ to $p_i$. This geodesic segment is transformed by $g_{ij}$ onto a geodesic segment from $g_{ij}p$ to $g_{ij}p_i=p_j\in \exp_p(B_p)$. We still have to show that $D(g_{ij}p)\in D(\exp_p(B_p))$.

By lemma \ref{lem-approxgeo}
\begin{align}
d_\cN(\rho(g_{ij})\widetilde\gamma(t_i),\widetilde\gamma(t_j)) &\leq r([c(t_j)]) \frac{4\eps-4\eps^2}{1-5\eps+3\eps^2}.
\end{align}
Now, take $B$ the ball centered in $\widetilde\gamma(t_j)$ and of radius $r([c(t_j)])$. Apply a dilatation of factor $r([c(t_j)])^{-1}$. It gives the unit ball $B(\widetilde\gamma(t_j),1)$. Translate this ball so it is centered in $0\in\cN$. By construction of the open balls (see theorem \ref{thm-Hebisch2}), this ball is given in coordinates by $\sum x_i^2 \leq 1$, where $(x_i)$ is the adapted basis of $\mathfrak n$ such that the dilatations act as diagonal transformations.
By a uniform choice of $\eps$, the corresponding point of $\rho(g_{ij})\widetilde\gamma(t_i)$ is really close to $0$.

Now we compare the geodesic segment corresponding to the geodesic segment $\widetilde\gamma(\{t\geq t_j\})$ with the one corresponding to $\rho(g_{ij})\widetilde\gamma(\{t\geq t_i\})$.
Both start very closely to $0$. Their tangent vector at $0$ is given  by the corresponding vector of $(1-t_j)w$, where $w=\dd D_p v_p$ has been transformed by parallel transport. The other one is given by $\rho(g_{ij})_*(1-t_i)w$. The linear transformation $\rho(g_{ij})_*$ has an orthogonal part close the the identity since we can suppose $j\gg i$ large enough. Therefore on each coordinate $w_i$ of $(w_1,\dots,w_n)=w$, we have $\rho(g_{ij})\simeq \lambda(g_{ij})w_i = \lambda(g_{ij})^{d_i}w_i$. It follows that if $w_i\geq 0$ then $\rho(g_{ij})_*w_i\geq 0$.

Therefore the geodesic segment $\rho(g_{ij})\widetilde\gamma(t_i)$ starts close to $\widetilde\gamma(t\geq t_i)$ and has its directional vector in the same spherical quadrant than $\widetilde\gamma(t\geq t_j)$. Therefore, the geodesic segment $\widetilde\gamma(t)$ is transformed by $\rho(g_{ij})$ onto a geodesic segment that still is in $D(\exp_p(B_p))$. In particular $\rho(g_{ij})\widetilde\gamma(0)=D(g_{ij}p)$ is in $D(\exp_p(B_p))$.
\end{proof}

The first part of the next proposition is inferred from the last discussion.
\begin{proposition}\label{fried-demiespace}
The exponential based in $p$ is well defined on a subspace of ${\rm T}_p\widetilde M$ given by
\begin{equation}
H_p = \bigcup_{j \gg i \gg 0} {\rm d} D^{-1}(\gamma_{ij}^{-1}({\rm d} D (B_p))).
\end{equation}
The subspace $H_p\subset \rT_p\widetilde M$ is a half-space. More precisely, the boundary $\partial D(\exp_p(H_p))$ is the left-translation of $\exp_0(V)$ with $V\subset \mathfrak n$ a codimension $1$ linear vector space.
\end{proposition}
\begin{proof}
Up to apply a left translation, suppose that $\widetilde\gamma(1)=0$. The ball $D(\exp_p(B_p))$ has a tangent hyperplane $T$ at $0$. Let $f_\lambda (x) = \lambda x$ be a dilatation of factor $\lambda>0$ small. Then $f_\lambda^{-1}(T)$ converges to a hyperplane $V$ when $\lambda \to 0$. This hyperplane needs not to be tangent to $D(\exp_p(B_p))$.

Now, consider a ball $B(0,\eps)$ with $\eps>0$ very small. The intersection $C=B(0,\eps)\cap D(\exp_p(B_p))$ ressembles to the intersection of $B(0,\eps)$ with the half space with boundary $T$. Let $H$ be the half-space with boundary $V$. Then any point $x\in H$ has its orbit $f_\lambda(x)$ converging to $C$. Therefore $x$ corresponds to a vector that is visible from $D(p)$.
\end{proof}

Note that $H$ is rarely the full developing image of the visible set from $p$. Indeed, $V$ might intersects $D(\exp_p(B_p))$. This indicates that the boundary of $H$ is sometimes visible.

\begin{corollary}
For each $p\in \widetilde M$, $H_p$ is convex.
\end{corollary}

We now divide $\partial H_p$ onto $W_p\sqcup I_p$, where $W_p$ denotes the visible vectors and $I_p$ the invisible vectors of $\partial H_p$.

\begin{lemma}\label{lem-Invaff}
The image $\exp_{D(p)}({\rm d}D_p(I_p))$ is locally constant (hence constant) following $p$, this image is denoted $I$. Furthermore, $I$ is affine.
\end{lemma}
By “affine” we will mean the following property: $I$ is equal to an intersection of a finite number of $\exp_{D(p)}(\dd D_p(\partial H_p))$.

\begin{proof}
If $W_p=\emptyset$, then $I_p=\partial H_p$ is affine and its image by the developing map must be constant since $\exp_p(V_p) = \widetilde M$ is convex (because $H_p$ is convex and equal to $V_p$ if $\partial H_p$ is only constituted of invisible vectors) and $D$ is therefore a diffeomorphism from $\exp_p(V_p)$ to its image.

Let $v_p\in I_p$. Suppose that $u\in W_p$. Let $q = \exp_p(u)$.
The point $q$ has a half-space $H_q$. We will show that $\exp_{D(p)}({\rm d}D_p(v_p))\in \exp_{D(q)}({\rm d}D_q(\partial H_q))$. It shows 
\begin{equation}
\exp_{D(p)}({\rm d}D_p(I_q))\subset \exp_{D(p)}({\rm d}D_p(\partial H_p))\cap \exp_{D(q)}({\rm d}D_q(\partial H_q)).
\end{equation}
Since $\exp_{D(q)}({\rm d}D_q(\partial H_q))\subsetneq \exp_{D(p)}({\rm d}D_p(\partial H_p))$, such an intersection decreases the topological dimension. By repeating the argument for a new $u$, the image of $I_p$ becomes constant following $p$ and is indeed affine.

Since $\exp_p(H_p)$ and $\exp_q(H_q)$ are convex and with non empty intersection, it follows that $D$ is injective on $(\exp_p(H_p))\cup (\exp_q(H_q))$.
If $\exp_{D(p)}({\rm d}D_p(v_p))$ lies in $D(\exp_q(H_q))$, then this is locally true: $D(c(t))=D(\exp_p(tv_p))\subset D(\exp_q(H_q))$ for $t\in ]T,1[$. By injectivity of $D$, this shows that $c(t)=\exp_p(tv_p)$ is contained in $\exp_q(H_q)$ for $t\in ]T,1[$ and this geodesic is defined for $t=1$ by hypothesis. But this contradicts the fact that $c(1)$ is not defined.
Therefore $\exp_{D(p)}({\rm d}D_p(v_p))\not\in D(\exp_q(H_q))$.

Take $p'$ in $\exp_p(H_p)$ such that $\exp_{D(p)}({\rm d}D_p(v_p))$ belongs to $\partial D(\exp_{p'}(B_{p'}))$. Then there exists $\rho(g_{ij})$ centered in $\exp_{D(p)}({\rm d}D_p(v_p))$ very contracting with almost no rotation for $i,j$ large enough.

If $\exp_{D(p)}({\rm d}D_p(v_p))\not\in\exp_{D(q)}({\rm d}D_q(\partial H_q))$, then this last set has no fixed point under $\rho(g_{ij})$. Therefore, $D(g_{ij}(\exp_{q}(H_q)))$ is convex and contains $\exp_{D(q)}({\rm d}D_q(\overline{B_q}))$. But $g_{ij}(\exp_{q}(H_q))$ and $\exp_q(H_q)$ intersect, hence $D$ is injective on the reunion of both.
If $c(t)\in \exp_q(B_q)$ is a geodesic such that $c(0)=q$ and $c(1)$ is not defined, then this shows that $c(t)$ is well defined in $t=1$, since it is in $g_{ij}(\exp_q(H_q))$, absurd.
\end{proof}

\begin{lemma}
The developing map $D\colon\widetilde M \to \cN-I$ is a covering map onto its image.
\end{lemma}
\begin{proof}Since the developing map $D$ is a local diffeomorphism, we only have to show the property of lifting paths.
Let $z\in D(\widetilde M)$. We have to show that if $\gamma\colon\iI \to D(\widetilde M)$ ends in $z$ and if it can be lifted to $\widetilde \gamma$ for $t<1$, then it can be lifted at $t=1$. Let $T<1$ be large enough. Then $r(\widetilde\gamma(T))$ is equal to the distance from $\gamma(T)$ to $I$. But for $T$ large enough, it is approximately equal to the distance of $z$ to $I$ and hence is minored by some constant $c>0$. We can choose $T$ such that $d(\gamma(T),z)< c/2<r(\widetilde\gamma(T))$. Therefore, $z$ belongs to the interior of $D(\exp_{\widetilde\gamma(T)}(B_{\widetilde \gamma(T)}))$. It follows that the path $\gamma$ can be lifted at $t=1$.
\end{proof}

Recall proposition \ref{prop-limdiscret}:
\begin{proposition*}[\ref{prop-limdiscret}]
Suppose that a subgroup $\Gamma\subset \Sim(\cN)$ is discrete. Let $f\in \Gamma$ and $a\in \cN$ be such that for any $x\in \cN$, we have $f^nx\to a$. Then for any $g\in \Gamma$, $g(a)=a$.
\end{proposition*}

Such maps $f$ and $g$ are given by the various $g_{ij}$ by changing the base point $p$. If $I$ is not a single point, then two such maps $f,g$ exist with two different attracting points. We will show that it does not occur.
If $\mathcal{N}-I$ is simply connected, then the holonomy group must be discrete, and with proposition \ref{prop-limdiscret} this raise a contradiction.

\paragraph{}
It remains to study the case when $\mathcal{N}-I$ is not simply connected. 
First we prove an intermediary lemma that explains how $I$ is impacted by the fact that it is invariant by dilatations. Let $(x_1,\dots,x_n)$ be a Lie algebra coordinates of $\cN$ such that a dilatation acts like $\delta_t (x_1,\dots,x_n) = (t^{d_1}x_1,\dots,t^{d_n}x_n)$. Up to conjugation by a translation, suppose that $I$ contains $0=(0,\dots,0)$. By lemma \ref{lem-Invaff}, $I$ is affine. Since $0\in I$, it implies that $I$ is given by linear equations (see proposition \ref{fried-demiespace}).

\begin{lemma}
Suppose that $\dim I=k>0$. Then any system of $n-k$ equations defining $I$ is such that if $a_1x_1+\dots+a_nx_n=0$ is  an equation, then every non-vanishing coordinates $a_i\neq 0$ are associated to a same degree $d=d_i$.
\end{lemma}
\begin{proof}
Let 
\begin{equation}
a_1x_1+\dots + a_nx_n=0
\end{equation}
be an equation verified by $I$. Then 
\begin{equation}
t^{d_1}a_1 x_1 + \dots + t^{d_n}a_n x_n=0
\end{equation}
is also verified by $I$ since it is stable by the dilatations $\delta_t$ ($I$ contains $0$). But the points associated to the coordinates $(t^{d_1}a_1,\dots,t^{d_n}a_n)$ with different values of $t>0$ are linearly independent if for $i\neq j$, $d_i\neq d_j$ and $a_i,a_j\neq 0$. Since $I$ has a non vanishing dimension $k>0$, it implies the lemma.
\end{proof}

The case where $\cN-I$ is not simply connected corresponds to the case where $I$ is of codimension $2$. Let $a_1x_1+\dots+a_nx_n=0$ and $b_1x_1+\dots+b_nx_n=0$ be two equations defining $I$. Let $H$ be $\{a_1x_1+\dots+a_nx_n> 0, b_1x_1+\dots+b_nx_n= 0\}$. Then $H$ is connected and $\partial H=I$. By the preceding lemma, $H$ is invariant by the dilatations $\delta_t$.

By rotating $H$ around $I$, we find that the universal cover of $\mathcal{N}-I$ is $H\times\mathbf{R}$. 
Let $\Sim(I)\subset\Sim(\cN)$ be the subgroup leaving invariant $I\subset \cN$. 
This subgroup contains the holonomy group $\Gamma$ since it leaves invariant $D(\widetilde M)=\cN-I$. Let $\Sim(I)_H$ be the subgroup of $\Sim(I)$ that preserves $H$.

\begin{lemma}
The subgroup $\Sim(I)_H$ acts transitively on $H$.
\end{lemma}
\begin{proof}Let $B\subset \cN$ be a small open ball at $0$. Let $B_H=B\cap H$.
The subgroup $\Sim(I)_H$ contains the dilatation subgroup $A\subset \Sim(\cN)$. Therefore, any point $p\in H$ can be assumed to be in $B_H$ up to apply an element of $A$. Furthermore, two points in $B_H$ can be assumed to be parallel along $I$, up to apply again a dilatation on one of them.

Let $N_I\subset N$ be the subgroup of the translations preserving $I$. Its lasts to show that $N_I\subset \Sim(I)_H$. A (small) translation acts like the flow of a vector field parallel to $I$. But such a vector field must preserve $I$ and (therefore) $\{b_1x_1+\dots+b_nx_n=0\}$.
Hence, this subgroup  sends $B_H$ in $H$. Otherwise it would intersect $I$ in the subspace $\{b_1x_1+\dots+b_nx_n=0\}$. But therefore $N_I$ sends $H$ in $H$ up to conjugation by $A$. This concludes the proof.
\end{proof}

It follows that the covering $H\times \R\to \cN-I$ has a lifted transformation group $\Sim(I)_H\times \R$ that acts transitively. If $f=g_{ij}$ then $\widetilde f\simeq f\times \{0\}$ since $f$ does not rotate much around $I$.
Therefore, we  have a lifting of $(G,X)$-structures
\begin{equation}
({\rm Sim}(I)_H\times \mathbf{R},H\times \mathbf{R})\to ({\rm Sim}(I),\mathcal{N}-I)
\end{equation}
and $M$ gets a lifted $({\rm Sim}(I)_H\times \mathbf{R},H\times \mathbf{R})$-structure. The developing map is given by the classic choice of a base point: take $p\in \widetilde M$ and $D(p)\in \mathcal N-I$, we have to choose $q\in H\times \mathbf R$ such that $q$ is send to $D(p)$ by the covering $H\times \mathbf R \to \mathcal N-I$, denote $q$ by $\widetilde D(p)$. The new developing map $\widetilde D$ is now fully prescribed by $D$, the point $\widetilde D(p)$ and the covering $H\times \R\to \cN-I$. Indeed, any point $x\in \widetilde M$ is seen as a path from $p$ to $x$ and $\widetilde D$ is chosen so this path developed by $D$ is lifted by the covering at $\widetilde D(p)$. 

Again, $\widetilde D$ is a covering map since $D$ is a covering map. 
Since $H\times \mathbf{R}$ is simply connected, the new holonomy group $\widetilde\Gamma$ must be discrete. But this is again contradicted by proposition \ref{prop-limdiscret} since we can take $\widetilde f,\widetilde g\in \widetilde \Gamma$ such that $\widetilde f \simeq f\times \{0\}$ and $\widetilde g\simeq g\times \{0\}$ with two different attracting points in $I=\partial H$.
This concludes the proof of theorem \ref{thm-nilpotent}.\qed

\section{Closed manifolds with a geometry modeled on the \\boundary of a rank one symmetric space}\label{sec-5}

Some classic facts can be derived from theorem \ref{thm-nilpotent}. The geometric structures arising as boundary geometries of rank one symmetric spaces are those  given by
$
\left( {\rm PU}_{\mathbf F}(n,1), \partial \mathbf H_{\mathbf F}^n \right)
$
where $\mathbf F$ can be the field of  real, or complex, or quaternionic or octonionic numbers. In the octonionic case, the only dimension considered is $n=2$. (See example 1 p.~\pageref{ex-1-iwasawa}.)

We will show the following result, which is mainly a consequence of theorem \ref{thm-nilpotent}. It seems that this was  known for at least the fields $\mathbf F=\mathbf R$ (see Kulkarni and Pinkall \cite{Kulkarni} and also Matsumoto \cite{Matsumoto}) and $\mathbf F= \mathbf C$ (see Falbel and Gusevskii \cite{Falbel}). The author does not know wether this was known for the quaternionic or octonionic field. Even if this result is not very new, our proof has the merit to be general since every such rank one boundary geometry is simultaneously taken into account.
\begin{theorem}\label{thm-rank1}
Let $M$ be a connected closed $\left({\rm PU}_{\mathbf F}(n,1),\partial \mathbf{H}^n_{\mathbf F}\right)$-manifold. 
If the developing map $D$ is not surjective then it is a covering onto its image. Furthermore,
 $D$ is a covering on its image if, and only if, $D(\widetilde M)$ is equal to a connected component of $\partial \mathbf{H}^n_{\mathbf F}-L(\Gamma)$, where $L(\Gamma)$ denotes the limit set of the holonomy group $\Gamma=\rho(\pi_1(M))$.
\end{theorem}

If we fix a point $\infty\in \partial \mathbf{H}^n_{\mathbf F}$, then we get a nilpotent similarity structure by taking $\left( {\rm PU}_{\mathbf F}(n,1)_\infty, \partial \mathbf H_{\mathbf F}^n-\{\infty\} \right)$ and we denote this structure by $(\Sim(\cN),\cN)$. (Compare again with example 1 p.~\pageref{ex-1-iwasawa}.)

\begin{definition}
Let $\Gamma\subset {\rm PU}_{\mathbf F}(n,1)$ be a subgroup. The \emph{limit set} $L(\Gamma)$ of $\Gamma$ is the subset of $\partial \mathbf{H}^n_{\mathbf F}$ given by 
\begin{equation}
L(\Gamma) = \partial \mathbf{H}^n_{\mathbf F}\cap \overline{\Gamma \cdot p}
\end{equation}
for any $p\in \mathbf{H}^n_{\mathbf F}$.
\end{definition}
\begin{lemma}
We have the following properties.
\begin{enumerate}
\item The definition of $L(\Gamma)$ does not depend on the choice of $p\in \mathbf H^n_{\mathbf F}$.
\item The set $L(\Gamma)$ is closed and invariant by the action of $\Gamma$.
\item If $C\subset \partial \mathbf{H}^n_{\mathbf F}$ is closed, $\Gamma$-invariant and contains at least two different points, then $L(\Gamma)\subset C$.
\item If $L(\Gamma)=\emptyset$ then $\Gamma\subset K$ in the $KAN$ decomposition of ${\rm PU}_{\mathbf F}(n,1)$.
\end{enumerate}
\end{lemma}
This lemma is well known. See for example the fundamental paper of Chen and Greenberg \cite{Chen}. Even if they don't consider the case where $\mathbf F$ is the octonionic field, everything remains true for our lemma. In particular, the last property  can  be shown by ${\rm CAT}(0)$-techniques, see for example  Bridson and Haefliger's book \cite[p. 179]{Bridson}.

Note that if $\Gamma$ is the holonomy group of a closed manifold $M$ and if $L(\Gamma)=\emptyset$, then $M$ is in fact a spherical manifold with a $(K, \partial \mathbf H_{\mathbf F}^n)$-structure and must be complete by compacity of $K$.

The following lemma can be found  in Kulkarni and Pinkall's paper \cite[theorem 4.2]{Kulkarni}, it was originally for the case where $\mathbf F=\R$. The name of this lemma comes from a paper of Falbel and Gusevskii \cite{Falbel}. The proof in \cite{Kulkarni} is easily extended to our general setting.
\begin{lemma}[Cutting lemma]\label{lem-cutting}
Let $M$ be a connected closed $\left({\rm PU}_{\mathbf F}(n,1),\partial \mathbf{H}^n_{\mathbf F}\right)$-manifold. Let $D$ be the developing map and $\Gamma$ be the holonomy group $\rho(\pi_1(M))$. If $L(\Gamma)$ consists of at least two points and if $D(\widetilde M)$ avoids $L(\Gamma)$, then $D$ is a covering map onto a connected component of $\partial \mathbf{H}^n_{\mathbf F}-L(\Gamma)$.
\end{lemma}

The theorem is cut into two parts given by the following propositions.
\begin{proposition}
Let $M$ be a connected closed $\left({\rm PU}_{\mathbf F}(n,1),\partial \mathbf{H}^n_{\mathbf F}\right)$-manifold. If the developing map $D$ is not surjective then it is a covering onto its image.
\end{proposition}
\begin{proof}
Let $\Omega = \partial \mathbf{H}^n_{\mathbf F}-D(\widetilde M)$. Then by hypothesis $\Omega\neq\emptyset$. Let $\Gamma$ be the holonomy group $\rho(\pi_1(M))$. By equivariance of the holonomy morphism and the developing map, $\Omega$ is $\Gamma$-invariant and is also  closed.
\begin{itemize}
\item If $\Omega=\{a\}$ then we can suppose that $a=\infty$ and then $\Gamma\subset \Sim(\cN)$. We also naturally have $D(\widetilde M)\subset \cN$. Therefore $M$ has a nilpotent similarity structure $(\Sim(\cN),\cN)$. Since $M$ is closed and since the developing map is surjective, this structure must be complete by theorem \ref{thm-nilpotent}.
\item Suppose that $\{a\}\subsetneq \Omega$. Then $L(\Gamma)\subset \Omega$. If $L(\Gamma)=\emptyset$, then $M$ is a spherical manifold $(K,\partial \mathbf{H}^n_{\mathbf F})$ and must be complete, absurd since $\Omega\neq \emptyset$. If $L(\Gamma) = \{b\}$, then again $M$ is a similarity manifold and by theorem \ref{thm-nilpotent}, $D$ is a covering map onto its image. If $L(\Gamma)$ consists of at least two points, then by the cutting lemma \ref{lem-cutting}, the developing map is a covering onto its image.
\end{itemize}
This shows the proposition.
\end{proof}
\begin{proposition}
Let $M$ be a connected closed $\left({\rm PU}_{\mathbf F}(n,1),\partial \mathbf{H}^n_{\mathbf F}\right)$-manifold.  The developing map $D$ is a covering on its image if, and only if, $D(\widetilde M)$ is equal to a connected component of $\partial \mathbf{H}^n_{\mathbf F}-L(\Gamma)$, where $L(\Gamma)$ denotes the limit set of the holonomy group $\Gamma=\rho(\pi_1(M))$.
\end{proposition}
\begin{proof}
Suppose that the developing map $D$ is a covering onto its image.
\begin{itemize}
\item If $D$ is a covering onto $\partial \mathbf{H}^n_{\mathbf F}$ or $\partial \mathbf{H}^n_{\mathbf F}-\{a\}$, then since both are simply connected, $D$ is a diffeomorphism. Therefore, the holonomy group $\Gamma$ is discrete and $D$ must avoid $L(\Gamma)$ since it is a covering onto its image. If $D$ is a covering onto $\partial \mathbf{H}^n_{\mathbf F}$ then $L(\Gamma)=\emptyset$ and $M$ is spherical. If $D$ is a covering onto $\partial \mathbf{H}^n_{\mathbf F}-\{a\}$ then $L(\Gamma)=\{a\}$ (otherwise $M$ would be spherical and complete).
\item If $D$ avoids at least two points, then this complement is closed and invariant and therefore contains $L(\Gamma)$. Also, $L(\Gamma)$ must contain at least two points, since otherwise $D$ would be a covering onto $\partial \mathbf H^n_{\mathbf F}$ or $\partial \mathbf{H}^n_{\mathbf F}-\{a\}$ by theorem \ref{thm-nilpotent}. By the cutting lemma \ref{lem-cutting}, $D$ is a covering onto a connected component of $\partial \mathbf{H}^n_{\mathbf F}-L(\Gamma)$.
\end{itemize}
Now we suppose that $D$ has its image equal to a connected component of $\partial \mathbf{H}^n_{\mathbf F}-L(\Gamma)$.
\begin{itemize}
\item If $L(\Gamma)=\emptyset$ then $M$ is spherical and $D$ is a covering map.
\item If $L(\Gamma)=\{a\}$ then $M$ is a nilpotent similarity manifold, and by hypothesis and theorem \ref{thm-nilpotent}, $D$ is a covering onto its image.
\item If $L(\Gamma)$ has at least two points, then by the cutting lemma \ref{lem-cutting}, the developing map is a covering onto its image.
\end{itemize}
It concludes the proof.
\end{proof}

It also concludes the proof of theorem \ref{thm-rank1}.\qed

\printbibliography

\end{document}